\documentclass{amsart}
\usepackage{amssymb}
\usepackage[mathcal]{euscript}
\usepackage{mathrsfs}
\usepackage[all,cmtip]{xy}

\theoremstyle{plain} 
\newtheorem{thm}[equation]{Theorem}
\newtheorem{cor}[equation]{Corollary}
\newtheorem{lem}[equation]{Lemma}
\newtheorem{prop}[equation]{Proposition}
\theoremstyle{definition}
\newtheorem{defn}[equation]{Definition}
\theoremstyle{remark}
\newtheorem{rem}[equation]{Remark}
\newtheorem{ex}[equation]{Example}

\begin{document}

\title[Stacks and sheaves of categories]
{stacks and sheaves of 
categories as fibrant objects}

\author{Alexandru E. Stanculescu}

\thanks{Supported by the project CZ.1.07/2.3.00/20.0003
of the Operational Programme Education for Competitiveness of 
the Ministry of Education, Youth and Sports of the Czech Republic}

\date{\today}

\address{\newline Department of Mathematics and Statistics,
\newline Masarykova Univerzita,  Kotl\'{a}{\v{r}}sk{\'{a}} 2,\newline
611 37 Brno, Czech Republic}
\email{stanculescu@math.muni.cz}

\begin{abstract}
We show that the category of categories
fibred over a site is a generalized Quillen 
model category in which the weak 
equivalences are the local equivalences and 
the fibrant objects are the stacks, as they were 
defined by J. Giraud. The generalized model 
category restricts to one on the full subcategory 
whose objects are the categories fibred in 
groupoids. We show that the category
of sheaves of categories is a 
model category that is Quillen 
equivalent to the generalized 
model category for stacks and 
to the model category for strong 
stacks due to A. Joyal and M. Tierney.
\end{abstract}
\maketitle
\section{Introduction}

The idea that stacks are the fibrant objects
of a model category was developed by 
A. Joyal and M. Tierney in \cite{JT}
and by S. Hollander in \cite{Ho1}.
The former paper uses internal 
groupoids and categories in a 
Grothendieck topos instead of 
fibred categories, and the latter 
only considers categories fibred in 
groupoids. The fibrant objects of 
the Joyal-Tierney model category 
are called strong stacks 
(of groupoids or categories), and 
the fibrant objects of Hollander's 
model category are the stacks 
of groupoids. Using some elaborate
results from the homotopy
theory of simplicial presheaves on 
a site, Hollander shows that her model
category is Quillen equivalent
to the model category for strong
stacks of groupoids. 

The purpose of this paper
is to extend Hollander's work to 
general stacks and to show that
the category of internal categories in 
a Grothendieck topos admits another 
model category structure that is Quillen 
equivalent to the model category 
for strong stacks of categories. Our 
approach is different from both \cite{Ho1} 
and \cite{JT}, and it was entirely inspired
by J. Giraud's book \cite{Gi}.
In fact, the influence of Giraud's 
work on ours cannot be overestimated.

Concerning general stacks, we give a 
realization of the thought that parts of 
Giraud's presentation of the theory 
of stacks \cite[Chapitre II \S1, \S2]{Gi} hint at a 
connection with left Bousfield localizations 
of model categories as presented by P.S.
Hirschhorn \cite[Chapter 3]{Hi}.
In more detail, let $E$ be a site, that is, a category 
$E$ equipped with a Grothendieck topology
and let {\rm Fib}$(E)$ be the category of fibred 
categories over $E$ and cartesian functors
between them. Let $\mathcal{C}$ be the class 
of maps $R\subset E_{/S}$ of {\rm Fib}$(E)$,
where $S$ ranges through the objects of $E$,
$E_{/S}$ is the category of objects of $E$ over 
$S$ and $R$ is a covering sieve (or, refinement) 
of $S$. Then Giraud's definition of stack resembles 
that of a $\mathcal{C}$-local object and his 
characterization of bicovering (\emph{bicouvrant} 
in French) maps resembles the $\mathcal{C}$-local 
equivalences of \cite[Definition 3.1.4(1)]{Hi}.
The bicovering maps are better known under 
the name `local equivalences'. 

The realization goes as follows. In order to deal 
with the absence of all finite limits and colimits 
in {\rm Fib}$(E)$ we introduce, following a 
suggestion of A. Joyal, the notion of 
generalized model category (see Definition 8). 
Many concepts and results from the theory 
of model categories can be defined in the 
same way and have an exact analogue for 
generalized model categories. We disregard 
that $E$ has a topology and we show that 
{\rm Fib}$(E)$ is naturally a generalized model 
category with the weak equivalences, cofibrations 
and fibrations defined on the underlying 
functors (see Theorem 13). Then
we show that `the left Bousfield localization 
of {\rm Fib}$(E)$ with respect to $\mathcal{C}$ 
exists', by which we mean that there is a
generalized model category structure on
{\rm Fib}$(E)$ having the bicovering maps
as weak equivalences and the stacks over $E$
as fibrant objects (see Theorem 29). We call this 
generalized model category the generalized 
model category for stacks over $E$ and we 
denote it by {\rm Champ}$(E)$.

To construct {\rm Champ}$(E)$ we make essential 
use of the functorial construction of the stack 
associated to a fibred category (or, stack completion) 
and some of its consequences \cite[Chapitre II \S2]{Gi}, 
and of a special property of bicovering maps 
(see Lemma 33). 

We adapt the method of proof of the existence 
of {\rm Champ}$(E)$ to show that 
{\rm Fibg}$(E)$, the full subcategory of 
{\rm Fib}$(E)$ whose objects are the categories 
fibred in groupoids, is a generalized model 
category in which the weak equivalences are 
the bicovering maps and the fibrant objects 
are the stacks of groupoids over $E$ 
(see Theorem 46). 

Concerning internal categories in a 
Grothendieck topos, let $\widetilde{E}$ 
be the category of sheaves on $E$. We 
show that the category {\rm Cat}$(\widetilde{E})$
of internal categories and internal functors
in $\widetilde{E}$ (or, sheaves of categories)
is a model category that is Quillen equivalent
to {\rm Champ}$(E)$ (see Theorem 48). 
We denote this model category by
{\rm Stack}$(\widetilde{E})_{proj}$.
The fibrant objects of 
{\rm Stack}$(\widetilde{E})_{proj}$
are the sheaves of categories that are
taken to stacks by the Grothendieck 
construction functor. To construct 
{\rm Stack}$(\widetilde{E})_{proj}$ we make 
essential use of the explicit way in which 
Giraud constructs the stack associated 
to a fibred category---a way that highlights the 
role of sheaves of categories, and of a variation of
Quillen's path object argument (see Lemma 49).
The model category {\rm Stack}$(\widetilde{E})_{proj}$
is also Quillen equivalent 
via the identity functors to the 
model category for strong stacks \cite[Theorem 4]{JT} 
(see Proposition 52) and it behaves 
as expected with respect to morphisms 
of sites (see Proposition 53).

The paper contains a couple of
other results, essentially easy 
consequences of some of the results
we have proved so far: the bicovering maps 
and the natural fibrations make 
{\rm Fib}$(E)$ a category of fibrant 
objects \cite{Br} (see Proposition 40), and 
the $2$-pullback (or, iso-comma object) 
of fibred categories is a model for the 
homotopy pullback in {\rm Champ}$(E)$
(see Lemma 42). 

Appendix 1 is a review of 
Hollander's characterization 
of stacks of groupoids in terms of 
the homotopy sheaf condition
\cite[Theorem 1.1]{Ho1}.
Appendix 2 studies the behaviour 
of left Bousfield localizations of 
model categories under change 
of cofibrations. The result contained 
in it is needed in Appendix 3,
which is a review of the model category 
for strong stacks of categories 
\cite[Theorem 4]{JT} made with the 
hope that it sheds some light on 
the nature of strong stacks.

I wish to express my gratitude to the referee 
whose comments and suggestions greatly 
improved the content of the paper. I wish to 
thank Jean B\'{e}nabou and Claudio Hermida 
for very useful correspondence related to 
fibred categories.

\section{Fibred categories}

In this section we recall, for 
completeness and to fix notations, some 
results from the theory of fibred categories. 

We shall work in the setting of universes, as 
in \cite{Gi}, although we shall not mention
the universe in which we shall be working.
We shall also use the axiom of choice.

We denote by $SET$ the category of sets and 
maps, by $CAT$ the category of categories 
and functors and by $GRPD$ its full subcategory 
whose objects are groupoids. 

Let $E$ be a category. We denote by $E^{op}$ 
the opposite category of $E$. We let $CAT_{/E}$ 
be the category of categories over $E$. Arrows 
of $CAT_{/E}$ will be called $E$-\emph{functors}.
If $S$ is an object of $E$, $E_{/S}$ stands for the 
category of objects of $E$ over $S$.

If $A$ and $B$ are two categories, 
we denote by $[A,B]$ the category of 
functors from $A$ to $B$ and natural 
transformations between them. 

We denote by $\ast$ the terminal 
object of a category, when it exists. We 
denote by $J$ the groupoid with two 
objects and one isomorphism between them.

\subsection{Isofibrations}
One says that a functor $A\to B$ is an 
\emph{isofibration} (called  
\emph{transportable} in \cite[Expos\'e VI]{Gr}) 
if it has the right lifting property with respect 
to one of the maps $\ast\to J$. A functor is 
both an isofibration and an equivalence of 
categories if and only if it is an equivalence 
which is surjective on objects (surjective 
equivalence, for short). Given a 
commutative diagram in $CAT$
\begin{displaymath}
\xymatrix{
{A}\ar[r] \ar[d]_{f}
&{C} \ar[d]^{g}\\
{B} \ar[r]
&{D}
}
\end{displaymath}
in which the horizontal arrows are 
surjective equivalences, if $f$ is 
an isofibration then so is $g$.

\subsection{Fibrations and isofibrations}
Let $E$ be a category.  

Let $f\colon F\to E$ be a functor. We denote by
$F_{S}$ the fibre category over $S\in Ob(E)$.
An $E$-functor $u\colon F\to G$ induces a 
functor $u_{S}\colon F_{S}\to G_{S}$
for every $S\in Ob(E)$.

\begin{lem} 
{\rm (1)} Let $u\colon F\to G$ be an $E$-functor
with $F\to E$ an isofibration. Then the
underlying functor of $u$ is an isofibration 
if and only if for every $S\in Ob(E)$, the 
map $F_{S}\to G_{S}$ is an isofibration.

{\rm (2)} Every fibration is an isofibration. 
Every surjective equivalence is a fibration.

{\rm (3)} Let $u\colon F\to G$ be an
$E$-functor such that the underlying functor
of $u$ is an equivalence. If $F$ is a fibration
then so is $G$.

{\rm (4)} Let $u\colon F\to G$ be an
$E$-functor such that the underlying functor
of $u$ is an equivalence. If $G$ is a fibration
and $F\to E$ is an isofibration then $F$ is a fibration.

{\rm (5)} Let $F$ and $G$ be two fibrations
and $u\colon F\to G$ be an $E$-functor. 
If the underlying functor of $u$ is full and 
faithful then $u$ reflects cartesian arrows.

{\rm (6)} Let 
\begin{displaymath}
\xymatrix{
{F}\ar[r] \ar[d]_{u}
&{H} \ar[d]^{v}\\
{G} \ar[r]
&{K}
}
\end{displaymath}
be a commutative diagram in $CAT_{/E}$
with $F,G,H$ and $K$ fibrations. If the
underlying functors of the horizontal arrows
are equivalences, then $u$ is a cartesian 
functor if and only if $v$ is cartesian.
\end{lem}
\begin{proof}
(1) We prove sufficiency. Let $\beta\colon u(x)\to y$
be an isomorphism and let $S=g(y)$. Then 
$g(\beta)\colon f(x)\to S$ is an isomorphism
therefore there is an isomorphism 
$\alpha\colon x\to x_{0}$ such that $f(\alpha)=g(\beta)$
since $f$ is an isofibration. The composite
$\beta u(\alpha^{-1})\colon u(x_{0})\to y$
lives in $G_{S}$ hence there is an isomorphism
$\alpha'\colon x_{0}\to x_{1}$ such that
$u(\alpha')=\beta u(\alpha^{-1})$ since
$u_{S}$ is an isofibration. One has 
$u(\alpha'\alpha)=\beta$.

(2) is straightforward. (3) and (4) are
consequences of \cite[Expos\'e VI Corollaire 
4.4 et Proposition 6.2]{Gr}.

(5) Let $f\colon F\to E$ and $g\colon G\to E$
be the structure maps. Let $\alpha\colon x\to y$ 
be a map of $F$ such that $u(\alpha)$ is cartesian.
We can factorize $\alpha$ as $c\gamma$, where
$c\colon z\to y$ is a cartesian map over $f(\alpha)$ 
and $\gamma\colon x\to z$ is a vertical map. 
Since $u(\alpha)$ is cartesian and $gu(\alpha)=
gu(c)$, there is a unique $\epsilon\colon u(z)\to u(x)$
such that $u(\alpha)\epsilon=u(c)$. Then 
$\epsilon=u(\beta)$ since $u$ is full, where 
$\beta\colon z\to x$. Hence $c=\alpha\beta$
since $u$ is faithful. Since $c$ is cartesian it
follows that $\gamma\beta$ is the identity,
and since $u(\alpha)$ is cartesian it follows
that $\beta\gamma$ is the identity. Thus, 
$\gamma$ is a cartesian map.

(6) is a consequence of (5) and
\cite[Expos\'e VI Corollaire 
4.4 et Proposition 5.3(i)]{Gr}.
\end{proof}
A \emph{sieve} of $E$
is a collection $R$ of objects of $E$ such that
for every arrow $X\to Y$ of $E$, $Y\in R$ implies
$X\in R$. Let $F\to E$ be a fibration and $R$ a sieve 
of $F$. The composite $R\subset F\to E$ is 
a fibration and $R\subset F$ is 
a cartesian functor.

A surjective equivalence takes sieves to sieves.
\subsection{The $2$-categories 
$\mathscr{F}ib(E)$ and $\mathscr{F}ibg(E)$}
Let $E$ be a category. We denote by 
{\rm Fib}$(E)$ the category whose objects 
are the categories fibred over $E$ and whose 
arrows are the cartesian functors. 

\subsubsection*{} 
Let $F$ and $G$ be two objects of {\rm Fib}$(E)$. 
The cartesian functors from $F$ to $G$ and 
the cartesian (sometimes called vertical) 
natural transformations between them form
a category which we denote by 
${\bf Cart}_{E}(F,G)$. This defines 
a functor
\begin{displaymath}
\xymatrix{
{{\bf Cart}_{E}(-,-)\colon 
{\rm Fib}(E)^{op}\times {\rm Fib}(E)}
\ar[r] &{CAT}
}
\end{displaymath}
so that the fibred categories over $E$,
the cartesian functors and 
cartesian natural transformations 
between them form a $2$-category
which we denote by $\mathscr{F}ib(E)$.

The category {\rm Fib}$(E)$ has finite products.
The product of two objects $F$ and $G$
is the pullback $F\times_{E}G$ .

Let $A$ be a category and $F\in {\rm Fib}(E)$.
We denote by $A\times F$ the pullback of
categories
\begin{displaymath}
\xymatrix{
{A\times F}\ar[r] \ar[d]
&{F} \ar[d]\\
{A\times E} \ar[r]
&{E}
}
\end{displaymath}
$A\times F$ is the product 
in {\rm Fib}$(E)$ of $F$
and $A\times E$.
The construction defines a functor
\begin{displaymath}
\xymatrix{
{-\times -\colon CAT\times {\rm Fib}(E)}
\ar[r] &{{\rm Fib}(E)}
}
\end{displaymath}
We denote by $F^{(A)}$ the pullback 
of categories
\begin{displaymath}
\xymatrix{
{F^{(A)}}\ar[r] \ar[d]
&{[A,F]} \ar[d]\\
{E} \ar[r]
&{[A,E]}
}
\end{displaymath}
so that $(F^{(A)})_{S}=[A,F_{S}]$.
The functor $-\times F$ is left adjoint
to ${\bf Cart}_{E}(F,-)$ and the functor
$A\times -$ is left adjoint to $(-)^{(A)}$.
These adjunctions are natural in $F$ and $A$.
There are isomorphisms that are 
natural in $F$ and $G$
\begin{equation}
{\bf Cart}_{E}(A\times F,G)\cong 
[A,{\bf Cart}_{E}(F,G)]\cong 
{\bf Cart}_{E}(F,G^{(A)})
\end{equation}
so that $\mathscr{F}ib(E)$ is
tensored and cotensored over 
the monoidal category $CAT$.

Let $F$ and $G$ be two objects of {\rm Fib}$(E)$.
We denote by $\mathsf{CART}(F,G)$ the
object of {\rm Fib}$(E)$ associated 
by the Grothendieck construction to
the functor $E^{op}\to CAT$ which
sends $S\in Ob(E)$ to 
${\bf Cart}_{E}(E_{/S}\times F,G)$,
so that
\begin{equation}
\mathsf{CART}(F,G)_{S}=
{\bf Cart}_{E}(E_{/S}\times F,G)
\end{equation}
There is a natural equivalence of categories
\begin{equation}
{\bf Cart}_{E}(F\times G,H)\simeq 
{\bf Cart}_{E}(F,\mathsf{CART}(G,H))
\end{equation}

\subsubsection*{}
The Grothendieck construction functor
\begin{displaymath}
\xymatrix{
{[E^{op},CAT]}\ar[r]^{\Phi}&{{\rm Fib}(E)}
}
\end{displaymath}
has a right adjoint $\mathcal{S}$
given by $\mathcal{S}F(S)={\bf Cart}_{E}(E_{/S},F)$.
$\Phi$ and $\mathcal{S}$ are $2$-functors
and the adjoint pair $(\Phi,\mathcal{S})$
extends to a $2$-adjunction betweeen 
the $2$-categories $[E^{op},CAT]$
and $\mathscr{F}ib(E)$. $\mathcal{S}F$ is a 
split fibration and $\mathcal{S}$ sends maps in 
{\rm Fib}$(E)$ to split functors. The composite 
$\mathsf{S}=\Phi\mathcal{S}$ sends 
fibrations to split fibrations and 
maps in {\rm Fib}$(E)$ to split functors.
The counit of the $2$-adjunction 
$(\Phi,\mathcal{S})$ is a $2$-natural 
transformation $\mathsf{v}\colon \mathsf{S}
\to Id_{\mathscr{F}ib(E)}$.
For every object $F$ of {\rm Fib}$(E)$ 
and every $S\in Ob(E)$ the map
\begin{equation}
(\mathsf{v}F)_{S}
\colon{\bf Cart}_{E}(E_{/S},F)\to F_{S}
\end{equation}
is a surjective equivalence.

$\Phi$ has also a left adjoint $\mathsf{L}$,
constructed as follows. For any category
$A$, the functor 
\begin{displaymath}
-\times A\colon CAT\to {\rm Fib}(A)
\end{displaymath}
has a left adjoint 
$\underset{\longrightarrow}{\rm Lim}(-/A)$
which takes $F$ to the category obtained
by inverting the cartesian morphisms
of $F$. If $F\in {\rm Fib}(E)$, 
\begin{displaymath}
\mathsf{L}F(S)=\underset{\longrightarrow}
{\rm Lim}(E^{/S}\times_{E}F/E^{/S})
\end{displaymath}
where $E^{/S}$ is the category of 
objects of $E$ under $S$. We denote
by $\mathsf{l}$ the unit of the adjoint
pair $(\mathsf{L},\Phi)$. For every 
$S\in Ob(E)$, the map 
$(\mathsf{l}F)_{S}\colon F_{S}
\to \mathsf{L}F(S)$ is an equivalence 
of categories. The adjoint pair
$(\mathsf{L},\Phi)$ extends to a 
$2$-adjunction betweeen the 
$2$-categories $[E^{op},CAT]$
and $\mathscr{F}ib(E)$.

\subsubsection*{} Let $F$ be an object of {\rm Fib}$(E)$.
We denote by $F^{cart}$ the subcategory of $F$ which
has the same objects and whose arrows are the cartesian
arrows. The composite $F^{cart}\subset F\to E$
is a fibration and $F^{cart}\subset F$ is a map in
{\rm Fib}$(E)$. For each $S\in Ob(E)$, $(F^{cart})_{S}$
is the maximal groupoid associated to $F_{S}$.
A map $u\colon F\to G$ of {\rm Fib}$(E)$
induces a map $u^{cart}\colon F^{cart}\to G^{cart}$ 
of {\rm Fib}$(E)$. In all, we obtain a functor
$(-)^{cart}\colon{\rm Fib}(E)\to {\rm Fib}(E)$.
One says that $F$ is \emph{fibred in groupoids} 
if the fibres of $F$ are groupoids. This is equivalent 
to saying that $F^{cart}=F$ and, if $f\colon F\to E$
is the structure map of $F$, to saying that
for every object $x$ of $F$,
the induced map $f_{/x}\colon F_{/x}\to
E_{/f(x)}$ is a surjective equivalence.

We denote by {\rm Fibg}$(E)$
the full subcategory of {\rm Fib}$(E)$ 
consisting of categories fibred in groupoids. 
The inclusion functor ${\rm Fibg}(E)\subset 
{\rm Fib}(E)$ has $(-)^{cart}$ as right adjoint.
{\rm Fibg}$(E)$ is a full subcategory of
$CAT_{/E}$.

We denote by $\mathscr{F}ibg(E)$ the full 
sub-$2$-category of $\mathscr{F}ib(E)$
whose objects are the categories fibred in 
groupoids. $\mathscr{F}ibg(E)$ is a 
$GRPD$-category. If $F$ and $G$ are 
two objects of $\mathscr{F}ibg(E)$,
we denote the $GRPD$-hom between $F$ 
and $G$ by ${\bf Cartg}_{E}(F,G)$.
$\mathscr{F}ibg(E)$ is tensored and 
cotensored over $GRPD$ with tensor
and cotensor defined by the same
formulas as for {\rm Fib}$(E)$.

$\mathscr{F}ib(E)$ becomes a 
$GRPD$-category by change of base along 
the maximal groupoid functor $max\colon 
CAT\to GRPD$. Then the inclusion 
$\mathscr{F}ibg(E)\subset \mathscr{F}ib(E)$
becomes a $GRPD$-functor which has
$(-)^{cart}$ as right $GRPD$-adjoint.
In particular, we have a natural 
isomorphism
\begin{equation}
{\bf Cartg}_{E}(F,G^{cart})\cong 
max{\bf Cart}_{E}(F,G)
\end{equation}

\subsection*{A change of base}
Let $m\colon A\to B$ be a functor.
There is a $2$-functor
\begin{displaymath}
m_{\bullet}^{fib}\colon
\mathscr{F}ib(B)\to \mathscr{F}ib(A)
\end{displaymath}
given by $m_{\bullet}^{fib}(F)=F\times_{B}A$.
If $A$ is fibred in groupoids with structure
map $m$, $m_{\bullet}^{fib}$ has a left
$2$-adjoint $m^{\bullet}$ that is given 
by composing with $m$. 

Let $P$ be a presheaf on $A$ 
and $D\colon SET\to CAT$
be the discrete category functor.
The functor $D$ induces a functor
$D\colon [E^{op},SET]\to [E^{op},CAT]$.
We denote the category $\Phi DP$
by $A_{/P}$, often called the 
category of elements of $P$.
As a consequence of the above 
$2$-adjunction we have a 
natural isomorphism
\begin{equation}
{\bf Cart}_{A}(A_{/P},m_{\bullet}^{fib}(F))
\cong {\bf Cart}_{B}(A_{/P},F)
\end{equation}
Let $E$ be a category
and $P$ a presheaf on $E$.
Let $m$ be the canonical map 
$E_{/P}\to E$. We denote 
$m_{\bullet}^{fib}(F)$ by
$F_{/P}$. As a consequence of 
the above $2$-adjunction we 
have a natural isomorphism
\begin{displaymath}
{\bf Cart}_{E_{/P}}(E_{/P},F_{/P})
\cong {\bf Cart}_{E}(E_{/P},F)
\end{displaymath}

\section{Generalized model categories}

\subsection{}
We shall need to work with a more 
general notion of (Quillen) model 
category than in the current literature
(like \cite{Hi}). In this section we shall
introduce the notion of generalized model 
category. Many concepts and results from 
the theory of model categories can be 
defined in the same way and have an 
exact analogue for generalized model 
categories. We shall review below some 
of them.
\begin{defn}
A \emph{generalized model category}
is a category $\mathcal{M}$ together
with three classes of maps {\rm W}, 
{\rm C} and {\rm F} (called 
\emph{weak equivalences}, 
\emph{cofibrations} and \emph{fibrations}) 
satisfying the following axioms:

A1: $\mathcal{M}$ has initial and terminal objects.

A2: The pushout of a cofibration along any map
exists and the pullback of a fibration along any
map exists.

A3: {\rm W} has the two out of three property.

A4: The pairs $({\rm C},{\rm F}\cap {\rm W})$ 
and $({\rm C}\cap {\rm W},{\rm F})$ are weak 
factorization systems.
\end{defn}
If follows from the definition that the 
classes {\rm C} and ${\rm C}\cap {\rm W}$
are closed under pushout and that the classes
{\rm F} and ${\rm F}\cap {\rm W}$ are closed 
under pullback. 

The opposite of the underlying 
category of a generalized model category
is a generalized model category.

Let $\mathcal{M}$ be a generalized 
model category. A map of $\mathcal{M}$ 
is a \emph{trivial fibration} if it is 
both a fibration and a weak equivalence,
and it is a \emph{trivial cofibration} if it is 
both a cofibration and a weak equivalence.
An object of $\mathcal{M}$
is \emph{cofibrant} if the map to it from 
the initial object is a cofibration, and it is
\emph{fibrant} if the map from it to the 
terminal object is a fibration. 
Let $X$ be an object of $\mathcal{M}$. 
For every cofibrant object $A$ of $\mathcal{M}$, 
the coproduct $A\sqcup X$ exists and the map 
$A\to A\sqcup X$ is a cofibration.
Dually, for every fibrant object $Z$, the 
product $Z\times X$ exists and the map 
$Z\times X\to X$ is a fibration.

The class of weak equivalences of a generalized
model category is closed under retracts 
\cite[Proposition 7.8]{JT2}.
\subsection{}
Let $\mathcal{M}$ be a generalized model 
category with terminal object $\ast$. Let 
$f\colon X\to Y$ be a map of $\mathcal{M}$ 
between fibrant objects. We review the
construction of the mapping path 
factorization of $f$ \cite{Br}. Let 
\begin{displaymath}
\xymatrix{
{Y}\ar[r]^{s} &{PathY} 
\ar[r]^{p_{0}\times p_{1}} & {Y\times Y}
}
\end{displaymath}
be a factorization of the diagonal 
map $Y\to Y\times Y$ into a weak 
equivalence $s$ followed by a fibration 
$p_{0}\times p_{1}$. Consider the 
following diagram
\begin{displaymath}
\xymatrix{
&{Pf}\ar[r]^{q} \ar[d]_{\pi_{f}}
&{X\times Y} \ar[d]_{f\times Y} 
\ar[r]^{p_{X}} &X\ar[d]^{f}\\
{Y} \ar[r]_{s} &
PathY \ar[r]_{p_{0}\times p_{1}} 
& Y\times Y \ar[r]_{p_{0}} \ar[d]_{p_{1}} 
& Y\ar[d]\\
& &Y\ar[r] &{\ast}
}
\end{displaymath}
in which all squares are pullbacks.
The object $Pf$ is fibrant.
There is a unique map $j_{f}\colon X\to Pf$
such that $\pi_{f}j_{f}=sf$ and 
$p_{X}qj_{f}=1_{X}$. The map
$p_{X}q$ is a trivial fibration, hence 
the map $j_{f}$ is a weak equivalence.
Put $q_{f}=p_{1}(f\times Y)q$. Then
$q_{f}$ is a fibration and $f=q_{f}j_{f}$.

In a generalized model category, the 
pullback of a weak equivalence between
fibrant objects along a fibration is a 
weak equivalence \cite[Lemma 2 on page 428]{Br}.
\subsection{}
Let $\mathcal{M}$ be a generalized 
model category. A \emph{left Bousfield 
localization} of $\mathcal{M}$ is a 
generalized model category 
{\rm L}$\mathcal{M}$ on the underlying 
category of $\mathcal{M}$ having the same 
class of cofibrations as $\mathcal{M}$ 
and a bigger class of weak equivalences. 
\begin{lem}
Let $\mathcal{M}$ be a generalized 
model category with {\rm W}, {\rm C} 
and {\rm F} as weak equivalences, 
cofibrations and fibrations. Let {\rm W}$'$
be a class of maps of $\mathcal{M}$ that
contains {\rm W} and has the two out
of three property. We define {\rm F}$'$
to be the class of maps having the 
right lifting property with respect
to every map of ${\rm C}\cap {\rm W'}$.

Then ${\rm L}\mathcal{M}=({\rm W}',
{\rm C},{\rm F}')$ is a left Bousfield 
localization of $\mathcal{M}$ if and only 
if the pair $({\rm C}\cap {\rm W}',F')$
is a weak factorization system. Moreover,
$({\rm C}\cap {\rm W}',F')$ is a weak 
factorization system if and only if
the class ${\rm C}\cap {\rm W}'$ is closed 
under codomain retracts and every arrow 
of $\mathcal{M}$ factorizes as a map 
in ${\rm C}\cap {\rm W'}$ 
followed by a map in {\rm F}$'$.
\end{lem}
\begin{proof}
We prove the first statement.
The necessity is clear. Conversely,
since ${\rm C}\cap {\rm W}\subset
{\rm C}\cap {\rm W}'$ it follows 
that ${\rm F}'\subset {\rm F}$. This 
implies that the second part of
Axiom A2 is satisfied. To complete
the proof it suffices to show that
${\rm F}\cap {\rm W}=
{\rm F}'\cap {\rm W}'$.
Since ${\rm C}\cap {\rm W}'\subset
{\rm C}$, it follows that 
${\rm F}\cap {\rm W}\subset
{\rm F}'$ and hence that 
${\rm F}\cap {\rm W}\subset
{\rm F}'\cap {\rm W}'$. We show
that ${\rm F}'\cap {\rm W}'\subset
{\rm F}\cap {\rm W}$. Let $X\to Y$ be 
a map in ${\rm F}'\cap {\rm W}'$.
We factorize it into a map $X\to Z$ in 
{\rm C} followed by a map $Z\to Y$ in 
${\rm F}\cap {\rm W}$. Since {\rm W}$'$
has the two out of three property,
the map $X\to Z$ is in ${\rm C}\cap {\rm W}'$.
It follows that the commutative diagram
\begin{displaymath}
\xymatrix{
{X}\ar@{=}[r]\ar[d]
&{X} \ar[d]\\
{Z} \ar[r]
&{Y}
}
\end{displaymath}
has a diagonal filler, hence
$X\to Y$ is a (domain) retract 
of $Z\to Y$. Thus, the map $X\to Y$
is in ${\rm F}\cap {\rm W}$.

The second statement follows
from a standard characterization
of weak factorization systems.
\end{proof}
Let {\rm L}$\mathcal{M}$ be a left 
Bousfield localization of $\mathcal{M}$.
A map of $\mathcal{M}$ between fibrant 
objects in {\rm L}$\mathcal{M}$ is a 
weak equivalence (fibration) in 
{\rm L}$\mathcal{M}$ if and only if 
it is a weak equivalence (fibration) 
in $\mathcal{M}$. Let $X\to Y$ be a
weak equivalence in $\mathcal{M}$
between fibrant objects in $\mathcal{M}$.
Then $X$ is fibrant in {\rm L}$\mathcal{M}$
if and only if $Y$ is fibrant in 
{\rm L}$\mathcal{M}$.
\subsection{}
A generalized model category
is \emph{left proper} if every pushout
of a weak equivalence along a cofibration
is a weak equivalence. Dually, a
generalized model category
is \emph{right proper} if every pullback
of a weak equivalence along a fibration
is a weak equivalence. A generalized 
model category is \emph{proper}
if it is left and right proper. A left 
Bousfield localization of a left
proper generalized model category 
is left proper.

Let $\mathcal{M}$ be a right 
proper generalized model category.
Let 
\begin{equation}
\vcenter{
\xymatrix{
{X}\ar[r]^{g} & {Z} & {Y}\ar[l]_{f}
}
}
\end{equation}
be a diagram in $\mathcal{M}$. 
We factorize $f$ as a trivial
cofibration $Y\overset{i_{f}}\to E(f)$
followed by a fibration
$E(f)\overset{p_{f}}\to Z$.
We factorize $g$ as a trivial
cofibration $X\overset{i_{g}}\to E(g)$
followed by a fibration
$E(f)\overset{p_{g}}\to Z$.
The \emph{homotopy pullback}
of diagram (10) is defined
to be the pullback of the diagram
\begin{displaymath}
\xymatrix{
{E(g)}\ar[r]^{p_{g}} & {Z} & {E(f)}\ar[l]_{p_{f}}
}
\end{displaymath}
The analogue of \cite[Proposition 13.3.4]{Hi}
holds in this context. If $X,Y$ and $Z$ are 
fibrant, a model for the homotopy pullback
is $X\times _{Z}Pf$, where $Pf$ is the 
mapping path factorization of $f$ described 
in Section 3.2.

Let {\rm L}$\mathcal{M}$ be a left 
Bousfield localization of $\mathcal{M}$
that is right proper. 
We denote by $X\times_{Z}^{h}Y$
the homotopy pullback in $\mathcal{M}$
of diagram (10) and by 
$X\times_{Z}^{{\rm L}h}Y$ the 
homotopy pullback of the same diagram, 
but in {\rm L}$\mathcal{M}$.

\begin{prop}
{\rm (1)} Suppose that $X,Y$ and $Z$ are 
fibrant in {\rm L}$\mathcal{M}$. Then 
$X\times_{Z}^{h}Y$  is weakly 
equivalent in $\mathcal{M}$  
to $X\times_{Z}^{{\rm L}h}Y$.

{\rm (2)} Suppose that the pullback
of a map between fibrant objects in 
$\mathcal{M}$ that is both a fibration in
$\mathcal{M}$ and a weak equivalence 
in {\rm L}$\mathcal{M}$ is a weak 
equivalence in {\rm L}$\mathcal{M}$.
Suppose that $X,Y$ and $Z$ are 
fibrant in $\mathcal{M}$. 
Then $X\times_{Z}^{h}Y$  is weakly 
equivalent in {\rm L}$\mathcal{M}$  
to $X\times_{Z}^{{\rm L}h}Y$.
\end{prop}
\begin{proof}
(1) is a consequence of 
\cite[Proposition 13.3.7]{Hi}.
To prove (2) we first factorize $f$ in 
{\rm L}$\mathcal{M}$ as a
trivial cofibration $Y\to Y_{0}$ 
in followed by a fibration $Y_{0}\to Z$. 
Then we factorize $Y\to Y_{0}$ in 
$\mathcal{M}$ as a trivial cofibration 
$Y\to Y'$ in followed by a fibration 
$Y'\to Y_{0}$. By assumption 
the map $X\times_{Z}Y'\to X\times_{Z}Y_{0}$ 
is a weak equivalence in 
{\rm L}$\mathcal{M}$.
\end{proof}
\subsection{}
Let $\mathcal{M}$ and $\mathcal{N}$
be generalized model categories
and $F\colon \mathcal{M}\to
\mathcal{N}$ be a functor having
a right adjoint $G$. The adjoint pair 
$(F,G)$ is a \emph{Quillen pair} if $F$ 
preserves cofibrations and trivial cofibrations. 
Equivalently, if $G$ preserves fibrations 
and trivial fibrations. If the classes of
weak equivalences of $\mathcal{M}$
and $\mathcal{N}$ have the two out 
of six property, then $(F,G)$ is a Quillen pair
if and only if $F$ preserves cofibrations 
between cofibrant objects and trivial 
cofibrations if and only if $G$ preserves 
fibrations between fibrant objects and 
trivial fibrations (a result due to Joyal).

The adjoint pair 
$(F,G)$ is a \emph{Quillen equivalence} 
if $(F,G)$ is a Quillen pair and if for every 
cofibrant object $A$ in $\mathcal{M}$ 
and every fibrant object $X$ in $\mathcal{N}$,
a map $FA\to X$ is a weak equivalence
in $\mathcal{N}$ if and only if its
adjunct $A\to GX$ is a weak
equivalence in $\mathcal{M}$.

\section{The natural generalized 
model category on {\rm Fib}$(E)$}

We recall \cite{JT} that $CAT$ is a model 
category in which the weak equivalences
are the equivalences of categories, the
cofibrations are the functors that are 
injective on objects and the fibrations 
are the isofibrations. Therefore,
for every category $E$, $CAT_{/E}$
is a model category in which a map
is a weak equivalence, cofibration 
or fibration if it is one in $CAT$.

Let $E$ be a category.
\begin{defn}
Let $u\colon F\to G$ be a map of 
{\rm Fib}$(E)$. We say that $u$ is an 
$E$-\emph{equivalence} (\emph{isofibration}) 
if the underlying functor of $u$ is an 
equivalence of categories (isofibration).
We say that $u$ is a \emph{trivial fibration} if 
it is both an $E$-equivalence and an isofibration.
\end{defn}
\begin{thm}
\label{first}
The category {\rm Fib}$(E)$ is a proper 
generalized model category with the 
$E$-equivalences as weak equivalences,
the maps that are injective on objects as 
cofibrations and the isofibrations as fibrations.
\end{thm}
The proof of Theorem 13 will be given
after some preparatory results.
\begin{prop}
Let $u\colon F\to G$ be a map of {\rm Fib}$(E)$. 
The following are equivalent:

{\rm (1)} $u$ is an $E$-equivalence.

{\rm (2)} For every $S\in Ob(E)$, the map 
$u_{S}\colon F_{S}\to G_{S}$ is an equivalence 
of categories.

{\rm (3)} $u$ is an equivalence in the 
$2$-category $\mathscr{F}ib(E)$.

{\rm (4)} ${\bf Cart}_{E}(u,X)\colon{\bf Cart}_{E}(G,X)
\to {\bf Cart}_{E}(F,X)$ is an 
equivalence for all $X\in {\rm Fib}(E)$.

{\rm (5)} ${\bf Cart}_{E}(X,u)\colon{\bf Cart}_{E}(X,F)
\to {\bf Cart}_{E}(X,G)$ is an 
equivalence for all $X\in {\rm Fib}(E)$.
\end{prop}
\begin{proof}
All is contained in \cite[Expos\'e VI]{Gr}.
\end{proof}
\begin{cor}
A map $u\colon F\to G$ of {\rm Fib}$(E)$
is a trivial fibration if and only if
for every $S\in Ob(E)$, 
$u_{S}\colon F_{S}\to G_{S}$ is 
a surjective equivalence.
\end{cor}
\begin{proof}
This follows from Lemma 1((1) and (2)) 
and Proposition 14.
\end{proof}
For part (2) of the next result, let
$\mathscr{M}$ be a class of functors
that is contained in the class of injective 
on objects functors. In our applications
$\mathscr{M}$ will be the class of injective
on objects functors or the set consisting
of one of the inclusions $\ast\to J$.
Let $\mathscr{M}^{\perp}$
be the class of functors that have the 
right lifting property with respect 
to every element of $\mathscr{M}$.
\begin{prop}
Let $u\colon F\to G$ be a map of {\rm Fib}$(E)$. 

{\rm (1)} $u$ is an isofibration if and 
only if $u^{cart}$ is an isofibration. 

{\rm (2)} If $u$ has the right lifting property
with respect to the maps $f\times E_{/S}$,
where $f\in \mathscr{M}$ and $S\in Ob(E)$, 
then $u_{S}\in \mathscr{M}^{\perp}$ 
for every $S\in Ob(E)$.
\end{prop}
\begin{proof}
(1) This is a consequence of Lemma 1(1)
and of the fact that for every $S\in Ob(E)$
and every object $F$ of {\rm Fib}$(E)$,
$(F^{cart})_{S}$ is the maximal groupoid
associated to $F_{S}$.

(2) Let $S\in Ob(E)$ and $A\to B$
be an element of $\mathscr{M}$.
Consider the commutative 
solid arrow diagram
\begin{displaymath}
\xymatrix@=4ex{
&&{{\bf Cart}_{E}(E_{/S},F)}\ar[dr]\ar[dd]\\
&&&{{\bf Cart}_{E}(E_{/S},G)}\ar[dd]\\
{A}\ar[rr]\ar[dr]\ar@{.>}[uurr]&&{F_{S}}\ar[dr]\\
&{B}\ar[rr]\ar@{.>}[uurr]&&{G_{S}}\\
}
\end{displaymath}
We recall that the category of arrows 
of $CAT$ is a model category in which
the weak equivalences and fibrations
are defined objectwise. A functor is
cofibrant in this model category if
and only if it is injective on objects.
If we regard the previous diagram
as a diagram in the category of 
arrows of $CAT$, then it has
by 2.3(5) and the assumption 
on $\mathscr{M}$ a diagonal filler,
the two dotted arrows. From
Section 2.3 and hypothesis
this diagonal filler has itself
a diagonal filler, hence the 
bottom face diagram has one.
\end{proof}
\begin{lem}
{\rm (1)}  Let
\begin{displaymath}
\xymatrix{
{F\times_{H}G}\ar[r] 
\ar[d]&{G} \ar[d]^{v}\\
{F} \ar[r]^{u}
&{H}
}
\end{displaymath}
be a pullback diagram in $CAT_{/E}$.
If $F,G$ and $H$ are fibrations,
$u$ and $v$ are cartesian functors
and $u$ is an isofibration, then 
$F\times_{H}G$ is a fibration
and the diagram is a pullback in 
{\rm Fib}$(E)$.

{\rm(2)} Let
\begin{displaymath}
\xymatrix{
{F}\ar[r]^{u}\ar[d]_{v}&{G}\ar[d]\\
{H} \ar[r]&{G\sqcup_{F}H}
}
\end{displaymath}
be a pushout diagram in $CAT_{/E}$.
If $F,G$ and $H$ are fibrations,
$u$ and $v$ are cartesian functors
and $u$ is injective on objects,
then $G\sqcup_{F}H$ is a fibration
and the diagram is a pushout in 
{\rm Fib}$(E)$.
\end{lem}
\begin{proof}
(1) The objects of $F\times_{H}G$
are pairs $(x,y)$ with $x\in Ob( F),y\in Ob(G)$
such that $u(x)=v(y)$. We shall briefly
indicate how the composite map 
$F\times_{H}G\to F\overset{p}\to E$ 
is a fibration. Let $S\in Ob(E)$, 
$(x,y)\in F\times_{H}G$ and $f\colon 
S\to p(x)$. A cartesian lift of $f$ is obtained 
as follows. Let $y^{f}\to y$ and $x^{f}\to x$ 
be cartesian lifts of $f$.
Since $H$ is a fibration, $u$ and $v$ are
cartesian functors and $u$ is an isofibration,
there is $x_{0}^{f}\in Ob(F_{S})$
such that $x^{f}\cong x_{0}^{f}$
and $u(x_{0}^{f})=v(y^{f})$.
Then the obvious map 
$(x_{0}^{f},y^{f})\to (x,y)$
is a cartesian lift of $f$.
The universal property of the 
pullback is easy to see. 

(2) The set of objects of 
$G\sqcup_{F}H$ can be identified
with $Ob(H)\sqcup (Ob(G)\setminus
ImOb(u))$. Since the structure functors
$G\to E$ and $H\to E$ are 
isofibrations, one can easily
check that the canonical 
map $G\sqcup_{F}H\to E$
is an isofibration. We shall
use Lemma 1(4) to show that
it is a fibration. Consider the 
following cube in $CAT_{/E}$
\begin{displaymath}
\xymatrix@=2ex{
{F}
\ar[rr]^{u}\ar[dr] \ar[dd]
&& {G} \ar[drr] \ar'[d][dd]\\
& {H} \ar[rrr] \ar[dd]
&&& {G\sqcup_{F}H} \ar[dd]\\
{\Phi\mathsf{L}F} \ar'[r][rr] \ar[dr]
&& {\Phi\mathsf{L}G} \ar[drr]\\
& {\Phi\mathsf{L}H} \ar[rrr]
&&& {\Phi\mathsf{L}G
\sqcup_{\Phi\mathsf{L}F}\Phi\mathsf{L}H}
}
\end{displaymath}
(see Section 2.3 for the functors
$\Phi$ and $\mathsf{L}$). The
top and bottom faces are
pushouts and the vertical
arrows having sources $F,G$ and 
$H$ are weak equivalences.
The map $\Phi\mathsf{L}u$ is a 
cofibration since $u$ is one. 
By \cite[Proposition 15.10.10(1)]{Hi}
the map
\begin{displaymath}
G\sqcup_{F}H\to \Phi\mathsf{L}G
\sqcup_{\Phi\mathsf{L}F}\Phi\mathsf{L}H
\end{displaymath}
is a weak equivalence. Since
$\Phi$ is a left adjoint, the
target of this map is in the image 
of $\Phi$, hence it is a fibration.
It follows from Lemma 1(4)
that $G\sqcup_{F}H$ is a fibration.
The canonical maps $H\to G\sqcup_{F}H$
and $G\to G\sqcup_{F}H$ are cartesian
functors by Lemma 1(6) applied
to the front and right faces of the 
above cube diagram. Finally, it 
remains to prove that if 
\begin{displaymath}
\xymatrix{
{F}\ar[r]^{u}\ar[d]_{v}&{G}\ar[d]\\
{H} \ar[r]&{K}
}
\end{displaymath}
is a commutative diagram in 
{\rm Fib}$(E)$, then the resulting
functor $G\sqcup_{F}H\to K$ is 
cartesian. This follows from 
Lemma 1(6) applied to the diagram
\begin{displaymath}
\xymatrix{
{G\sqcup_{F}H}\ar[r]\ar[d]&
{\Phi\mathsf{L}G
\sqcup_{\Phi\mathsf{L}F}\Phi\mathsf{L}H}
\ar[d]\\
{K} \ar[r]&{\Phi\mathsf{L}K}
}
\end{displaymath}
\end{proof}
\begin{rem}
A consequence of Lemma 17(1) is that
the fibre category $(F\times_{H}G)_{S}$ is 
the pullback $F_{S}\times_{H_{S}}G_{S}$.
Thus, if $F,G$ and $H$ are fibred in groupoids
then so is $F\times_{H}G$. A consequence 
of Lemma 17(2) is that if $F,G$ and $H$ are 
fibred in groupoids then so is $G\sqcup_{F}H$.
\end{rem}
\begin{ex}
(1) Let $u\colon F\to G$ be a map
of {\rm Fib}$(E)$ and $H$ an
object of {\rm Fib}$(E)$. Then
the diagram
\begin{displaymath}
\xymatrix{
{F\times H}\ar[r]^{F\times u}
\ar[d]&{G\times H} \ar[d]\\
{F} \ar[r]^{u}
&{G}
}
\end{displaymath}
in a pullback in {\rm Fib}$(E)$.

(2) Let 
\begin{displaymath}
E_{/}\colon E\to CAT_{/E}
\end{displaymath}
be the functor which takes 
$S$ to $E_{/S}$. The functor $E_{/}$
preserves all the limits that
exist in $E$. Therefore, if
\begin{displaymath}
\xymatrix{
{U\times_{S}T}\ar[r] 
\ar[d]&{T} \ar[d]\\
{U} \ar[r]
&{S}
}
\end{displaymath}
is a pullback diagram in $E$, 
then
\begin{displaymath}
\xymatrix{
{E_{/U\times_{S}T}}\ar[r] 
\ar[d]&{E_{/T}} \ar[d]\\
{E_{/U}} \ar[r]
&{E_{/S}}
}
\end{displaymath}
is a pullback diagram in 
{\rm Fib}$(E)$.
\end{ex}
\begin{proof}[Proof of Theorem ~\ref{first}]
Axioms A1 and A3 from Definition 8
are clear. Axiom A2 was dealt with 
in Lemma 17. We prove Axiom 4.
Any map $u\colon F\to G$ of {\rm Fib}$(E)$
admits a factorization $u=vi\colon F\to H\to G$
in $CAT_{/E}$, where $i$ is injective on
objects and the underlying functor of $v$
is a surjective equivalence. By Lemma 1(2)
$H$ is an object of {\rm Fib}$(E)$. By Lemma 
1(5) $i$ is a cartesian functor. By
\cite[Expos\'e VI Proposition 5.3(i)]{Gr}
$v$ is a cartesian functor. Any 
commutative diagram
\begin{displaymath}
\xymatrix{
{F}\ar[r] \ar[d]_{u}
&{H} \ar[d]^{v}\\
{G} \ar[r]
&{K}
}
\end{displaymath}
{\rm Fib}$(E)$ in which 
the underlying functor
of $u$ is injective on objects
and $v$ is a trivial fibration
has a diagonal filler in $CAT_{/E}$.
By Lemma 1(5) (or 
\cite[Expos\'e VI Corollaire 5.4]{Gr}, 
for example) this diagonal filler
is a cartesian functor. Thus, the
first part of Axiom 4 is proved.
The rest of the Axiom 4 is
proved similarly, using 
Lemma 1((3) and (5)) and
\cite[Expos\'e VI Proposition 5.3(i)]{Gr}.

Properness is easy to see.
\end{proof}
\begin{rem}
Let $F$ be an object of {\rm Fib}$(E)$. 
Let $D2$ be the discrete category
with two objects. By cotensoring the 
sequence $D2\to J\to \ast$ with $F$ 
we obtain a natural factorization 
of the diagonal $F\to F\times F$ as
\begin{displaymath}
\xymatrix{
{F}\ar[r] &{F^{(J)}} \ar[r] & {F\times F}
}
\end{displaymath}
in which the map $F\to F^{(J)}$ is an 
$E$-equivalence and the map 
$F^{(J)}\to F\times F$ is an isofibration.
We obtain the following model for the 
mapping path factorization (Section 3.2)
of a map $u\colon F\to G$ of {\rm Fib}$(E)$. 
The objects of a fibre category $(Pu)_{S}$ 
are triples $(x,y,\theta)$ with 
$x\in Ob(F_{S}),y\in Ob(G_{S})$ and 
$\theta \colon y\to u(x)$ an isomorphism in 
$G_{S}$. The arrows are pairs of arrows 
making the obvious diagram commute.
\end{rem}
\begin{prop}[Compatibility with the $2$-category 
structure] Let $u\colon F\to G$ be a cofibration 
and $v\colon H\to K$ an isofibration. Then the 
canonical map
\begin{displaymath}
\xymatrix{
{{\bf Cart}_{E}(G,H)} \ar[r] &{{\bf Cart}_{E}(G,K)
\times_{{\bf Cart}_{E}(F,K)} {\bf Cart}_{E}(F,H)}
}
\end{displaymath}
is an isofibration that is a surjective 
equivalence if either $u$ or $v$ is an 
$E$-equivalence.
\end{prop}
\begin{proof}
By Section 2.3 the diagram
\begin{displaymath}
\xymatrix{
{\ast} \ar[r] \ar[d] &{{\bf Cart}_{E}(G,H)}\ar[d]\\
{J} \ar[r] &{{\bf Cart}_{E}(G,K)\times_{{\bf Cart}_{E}(F,K)} 
{\bf Cart}_{E}(F,H)}
}
\end{displaymath}
has a diagonal filler if and only if the diagram
\begin{displaymath}
\xymatrix{
{F} \ar[r] \ar[d] &{H^{(J)}}\ar[d]\\
{G} \ar[r] &{K^{(J)}\times_{K}H}
}
\end{displaymath}
has one (the pullback exists
by Lemma 17(1)).
The latter is true since the map 
$H^{(J)}\to K^{(J)}\times_{K}H$ is a 
trivial fibration using Corollary 15.
Suppose that $u$ is an $E$-equivalence.
By Proposition 14, the functors 
${\bf Cart}_{E}(u,H)$ and ${\bf Cart}_{E}(u,K)$ 
are surjective equivalences. Since
surjective equivalences are stable 
under pullback, the functor 
\begin{displaymath}
{\bf Cart}_{E}(G,H)\to{\bf Cart}_{E}(G,K)
\times_{{\bf Cart}_{E}(F,K)} {\bf Cart}_{E}(F,H)
\end{displaymath}
is an equivalence by the two
out of three property of equivalences.
Suppose that $v$ is an 
$E$-equivalence. Then ${\bf Cart}_{E}(F,v)$ 
and ${\bf Cart}_{E}(G,v)$ are equivalences
the functor
\begin{displaymath}
\xymatrix{
{{\bf Cart}_{E}(G,K)\times_{{\bf Cart}_{E}(F,K)} 
{\bf Cart}_{E}(F,H)} \ar[r]&{{\bf Cart}_{E}(G,K)}
}
\end{displaymath}
is an equivalence being the pullback 
of an equivalence along an isofibration. 
Therefore the canonical map is an 
equivalence.
\end{proof}
\begin{cor}
A map $u\colon F\to G$ is an isofibration
if and only if for every object $X$ of 
{\rm Fib}$(E)$, the map
$${\bf Cart}_{E}(X,u)\colon{\bf Cart}_{E}(X,F)
\to {\bf Cart}_{E}(X,G)$$
is an isofibration.
\end{cor}
\begin{proof}
One half is a consequence of Proposition 
21. The other half follows by putting
$X=E_{/S}$, where $S\in Ob(E)$,
and using  2.3(5), Lemma 1(1) 
and Section 2.1.
\end{proof}
\begin{cor}[Compatibility with the `internal hom']
Let $u:F\to G$ be a cofibration and
$v:H\to K$ an isofibration. Then the 
canonical map
\begin{displaymath}
\xymatrix{
{\mathsf{CART}(G,H)}\ar[r] &{\mathsf{CART}(G,K)
\times_{\mathsf{CART}(F,K)} \mathsf{CART}(F,H)}
}
\end{displaymath}
is an isofibration that is a trivial fibration
if either $u$ or $v$ is an $E$-equivalence.
\end{cor}
\begin{proof}
The map $\mathsf{CART}(u,K)$ is
an isofibration by 2.3(3) and Proposition 
21, therefore the pullback in the displayed 
arrow exists by Lemma 17(1).
The result follows from Remark 18, 2.3(3) 
and Proposition 21 applied to $v$ and the 
cofibration $E_{/S}\times u$, $S\in Ob(E)$.
\end{proof}
We recall \cite[Theorem 4]{JT} that the 
category $[E^{op},CAT]$ is a model 
category in which a map is a weak
equivalence or cofibration if it is
objectwise an equivalence of categories
or objectwise injective on objects. We 
denote this model category by 
$[E^{op},CAT]_{inj}$. We recall that the 
category $[E^{op},CAT]$ is a model 
category in which a map is a weak
equivalence or fibration if it is
objectwise an equivalence of categories
or objectwise an isofibration. We 
denote this model category by 
$[E^{op},CAT]_{proj}$. 
The identity functors form a Quillen
equivalence between $[E^{op},CAT]_{proj}$
and $[E^{op},CAT]_{inj}$.

Recall from Section 2.3 the adjoint pairs
$(\Phi,\mathcal{S})$ and $(\mathsf{L},\Phi)$.
\begin{prop}
The adjoint pair $(\Phi,\mathcal{S})$
is a Quillen equivalence between {\rm Fib}$(E)$
and $[E^{op},CAT]_{inj}$.
The adjoint pair $(\mathsf{L},\Phi)$ is a
Quillen equivalence between {\rm Fib}$(E)$
and $[E^{op},CAT]_{proj}$.
\end{prop}
\begin{proof}
The functor $\Phi\colon [E^{op},CAT]_{inj}
\to {\rm Fib}(E)$ preserves and reflects 
weak equivalences and preserves cofibrations.
Since the map $\mathsf{v}F$ 
is a weak equivalence (2.3(5)), the pair 
$(\Phi,\mathcal{S})$ is a Quillen 
equivalence. 

The functor $\Phi\colon [E^{op},CAT]_{proj}
\to {\rm Fib}(E)$ preserves fibrations.
Since the map $\mathsf{l}F$ 
is a weak equivalence, the pair 
$(\mathsf{L},\Phi)$ is a Quillen 
equivalence.
\end{proof}
Let $m\colon A\to B$ be a category
fibred in groupoids. Recall from 
Section 2.3 that the functor $m_{\bullet}^{fib}
\colon{\rm Fib}(B)\to {\rm Fib}(A)$
has a left adjoint $m^{\bullet}$.
The proof of the next result is straightforward.
\begin{prop}
Let $m\colon A\to B$ be a category
fibred in groupoids. The adjoint
pair $(m^{\bullet},m_{\bullet}^{fib})$
is a Quillen pair.
\end{prop}
Let $f\colon T\to S$ be a map 
of $E$. The functor $f_{\bullet}^{fib}
\colon{\rm Fib}(E_{/S})\to {\rm Fib}(E_{/T})$
has a left adjoint $f^{\bullet}$.
\begin{cor}
Let $f\colon T\to S$ be a map 
of $E$. The adjoint pair 
$(f^{\bullet},f_{\bullet}^{fib})$
is a Quillen pair.
\end{cor}

\section{The generalized model 
category for stacks over a site}
We briefly recall from 
\cite[Chapitre 0 D\'{e}finition 1.2]{Gi} 
the notion of site. Let $E$ be a category. 
A \emph{topology} on $E$ is an 
application which associates to each 
$S\in Ob(E)$ a non-empty collection 
$J(S)$ of sieves of $E_{/S}$. This 
data must satisfy two axioms. The 
elements of $J(S)$ are called 
\emph{refinements} of $S$. A \emph{site}
is a category endowed with a topology.

Every category $E$ has 
the \emph{discrete} topology (only 
$E_{/S}$ is a refinement of the object 
$S$) and the \emph{coarse} topology 
(every sieve of $E_{/S}$ is a refinement 
of $S$). Any other topology on $E$ is
`in between' the discrete one and the 
coarse one.

Let $E$ be a site. Let $\mathcal{C}$ 
be the collection of maps 
$R\subset E_{/S}$ of {\rm Fib}$(E)$, 
where $S$ ranges through $Ob(E)$ 
and $R$ is a refinement of $S$.

Since $CAT$ is a model category, 
the theory of homotopy fiber squares
\cite[Section 13.3.11]{Hi} is available.
\begin{defn}
A map $F\to G$ of {\rm Fib}$(E)$ 
\emph{has property} $P$ if for every 
element $R\subset E_{/S}$ of
$\mathcal{C}$, the diagram
\begin{displaymath}
\xymatrix{
{{\bf Cart}_{E}(E_{/S},F)}\ar[r] \ar[d]
&{{\bf Cart}_{E}(R,F)}\ar[d]\\
{{\bf Cart}_{E}(E_{/S},G)}\ar[r]
&{{\bf Cart}_{E}(R,G)}
}
\end{displaymath}
in which the horizontal arrows are the 
restriction functors, is a homotopy fiber 
square. The map $F\to G$  is a 
$\mathcal{C}$-\emph{local fibration}
if it is an isofibration and it has property 
$P$. An object $F$ of {\rm Fib}$(E)$ is 
$\mathcal{C}$-\emph{local} if the map 
$F\to E$ is a $\mathcal{C}$-local fibration.
The map $F\to G$ is a 
$\mathcal{C}$-\emph{local equivalence} 
if for all $\mathcal{C}$-local objects $X$, 
the map 
\begin{displaymath}
{\bf Cart}_{E}(u,X)\colon{\bf Cart}_{E}(G,X)\to 
{\bf Cart}_{E}(F,X)
\end{displaymath}
is an equivalence of categories. 
\end{defn}
It follows directly from Definition 27 
and a standard property of homotopy 
fiber squares that a $\mathcal{C}$-local
object is the same as a 
stack (=($E$-)\emph{champ}) 
in the sense of 
\cite[Chapitre II D\'efinition 1.2.1(ii)]{Gi}. 
\begin{ex}
We shall recall that `sheaves are stacks'.

Let $\widehat{E}$ be the category 
of presheaves on $E$ and $\eta$ be
the Yoneda embedding. Let 
$D\colon SET\to CAT$ denote the
discrete category functor; it
induces a functor $D\colon 
\widehat{E}\to [E^{op},CAT]$.
For every objects $X,Y$ of 
$\widehat{E}$
there is a natural isomorphism
\begin{displaymath}
{\bf Cart}_{E}(\Phi DX,\Phi DY)\cong
D{\rm Fib}(E)(\Phi DX,\Phi DY)
\end{displaymath}
The composite functor 
$\Phi D\colon\widehat{E}\to {\rm Fib}(E)$ 
is full and faithful, hence we obtain
a natural isomorphism
\begin{displaymath}
{\bf Cart}_{E}(\Phi DX,\Phi DY)\cong
D\widehat{E}(\Phi DX,\Phi DY)
\end{displaymath}
Let now $S\in Ob(E)$ and $R$ be a 
refinement of $S$. Let $R'$ be 
the sub-presheaf of 
$\eta(S)$ which corresponds
to $R$. Since $E_{/S}=\Phi D\eta(S)$ 
and $R=\Phi DR'$, the previous
natural isomorphism shows
that a presheaf $X$ on $E$ is a sheaf 
if and only if $\Phi DX$ is a stack. In 
particular, $\eta(S)$ is a sheaf if 
and only if $E_{/S}$ is a stack.
\end{ex}
\begin{thm}
\label{second}
There is a proper generalized model category 
{\rm Champ}$(E)$ on the category 
{\rm Fib}$(E)$ in which the weak 
equivalences are the $\mathcal{C}$-local 
equivalences and the cofibrations are the 
maps that are injective on objects. The fibrant 
objects of {\rm Champ}$(E)$ are the stacks.
\end{thm}
The proof of Theorem 29 will be 
given after some preparatory results.
\begin{prop}
{\rm (1)} Every $E$-equivalence is a 
$\mathcal{C}$-local equivalence.

{\rm (2)} The class of maps having property $P$ 
is invariant under $E$-equivalences.

{\rm (3)} The class of maps having property $P$ 
contains $E$-equivalences and all maps
between stacks.

{\rm (4)} The class of maps having property $P$ 
is closed under compositions, pullbacks 
along isofibrations and retracts.
\end{prop}
\begin{proof}
(1) follows from Proposition 14. (2) says 
that for every commutative diagram 
\begin{displaymath}
\xymatrix{
{F}\ar[r] \ar[d]_{u}
&{H} \ar[d]^{v}\\
{G} \ar[r]
&{K}
}
\end{displaymath}
in which the horizontal maps are $E$-equivalences,
$u$ has property $P$ if and only if $v$ has it. This is 
so by Proposition 14 and \cite[ Proposition 13.3.13]{Hi}.
(3) follows from a standard property of homotopy
fiber squares. (4) follows from standard properties 
of homotopy fiber squares and the fact that 
equivalences are closed under retracts.
\end{proof}
\begin{lem}
A map between stacks has the right lifting property 
with respect to all maps that are both cofibrations 
and $\mathcal{C}$-local equivalences if and only 
if it is an isofibration.
\end{lem}
\begin{proof}[Proof (sufficiency)]
Let $H\to K$ be an isofibration 
between stacks and $F\to G$ 
a map that is both a cofibration 
and a $\mathcal{C}$-local 
equivalence. A commutative diagram
\begin{displaymath}
\xymatrix{
{F}\ar[r] \ar[d]
&{H} \ar[d]\\
{G} \ar[r]
&{K}
}
\end{displaymath}
has a diagonal filler if and 
only if the functor
\begin{displaymath}
{\bf Cart}_{E}(G,H)\to{\bf Cart}_{E}(G,K)
\times_{{\bf Cart}_{E}(F,K)} {\bf Cart}_{E}(F,H)
\end{displaymath}
is surjective on objects. We show
that it is a surjective equivalence.
The functor is an isofibration by 
Proposition 21. Hence it suffices
to show that it is an equivalence.
The maps ${\bf Cart}_{E}(G,K) \to
{\bf Cart}_{E}(F,K)$ and
${\bf Cart}_{E}(G,H) \to
{\bf Cart}_{E}(F,H)$
are surjective equivalences 
by assumption and Proposition 21. 
Since surjective equivalences are 
stable under pullback, the required 
functor is an equivalence by the two
out of three property of equivalences.
\end{proof}
For the notion of bicovering (=\emph{bicouvrant}) 
map in {\rm Fib}$(E)$ we refer the reader to
\cite[Chapitre II D\'efinition 1.4.1]{Gi}. As in 
[\emph{loc. cit.}, Chapitre II 1.4.1.1], we 
informally say that a map is bicovering if 
it is `locally bijective on arrows' and 
`locally essentially surjective on objects'.
\begin{ex}
For every $S\in Ob(E)$ and every 
refinement $R$ of $S$, $R\subset E_{/S}$ 
is a bicovering map.
\end{ex}
By \cite[Chapitre II Proof of Th\'eor\`eme 
d'existence 2.1.3]{Gi} there are a $2$-functor 
$\mathsf{A}\colon \mathscr{F}ib(E)\to \mathscr{F}ib(E)$
and a $2$-natural transformation $a:Id_{\mathscr{F}ib(E)}
\to \mathsf{A}$ such that $\mathsf{A}F$ is a stack and 
$aF$ is bicovering for every object $F$ of $\mathscr{F}ib(E)$.
By \cite[Chapitre II Corollaire 2.1.4]{Gi} the class
of bicovering maps coincides with the class of 
$\mathcal{C}$-local equivalences in the sense 
of Definition 27.
\begin{lem}
Bicovering maps are closed
under pullbacks along isofibrations.
\end{lem}
\begin{proof}
Let 
\begin{displaymath}
\xymatrix{
{F\times_{H}G}\ar[r]^-{u'} 
\ar[d]&{G} \ar[d]^{v}\\
{F} \ar[r]^{u}
&{H}
}
\end{displaymath}
be a pullback diagram
in {\rm Fib}$(E)$ with
$v$ an isofibration
(see Lemma 17(1)).

\emph{Step 1.} Suppose that
the above pullback diagram 
is a pullback diagram
of split fibrations and split functors
with $v$ an arbitrary split functor and 
$u$ `locally bijective on arrows'. We prove 
that $u'$ is `locally bijective on arrows'.
Let $S\in Ob(E)$ and $(x,y),(x',y')$ be two
objects of $(F\times_{H}G)_{S}$. Then, 
in the notation of \cite[Chapitre I 2.6.2.1]{Gi} 
and the terminology of 
\cite[Chapitre 0 D\'efinition 3.5]{Gi}
we have to show that the map 
\begin{displaymath}
\xymatrix{
{{\rm Hom}_{S}((x,y),(x',y'))} 
\ar[r] &{{\rm Hom}_{S}(y,y')}
}
\end{displaymath}
of presheaves on $E_{/S}$ is 
bicovering, where $E_{/S}$
has the induced topology 
\cite[Chapitre 0 3.1.4]{Gi}.
This map is the pullback of the map
\begin{displaymath}
\xymatrix{
{{\rm Hom}_{S}(x,x')} 
\ar[r] &{{\rm Hom}_{S}(u(x),u(x'))}
}
\end{displaymath}
which is by assumption bicovering.
But bicovering maps of presheaves 
are stable under pullbacks  
\cite[Chapitre 0 3.5.1]{Gi}.

\emph{Step 2.}
Suppose that in the above 
pullback diagram the map $u$ 
is `locally essentially surjective 
on objects'. We prove that $u'$
is `locally essentially surjective 
on objects'. Let $S\in Ob(E)$
and $y\in Ob(G_{S})$.
Let $R'$ be the set of maps 
$f\colon T\to S$ such that 
there are $x\in Ob(F_{T})$ and
$y'\in Ob(G_{T})$ with
$u_{T}(x)=v_{T}(y')$ and 
$y'\cong f^{\ast}(y)$ in $G_{T}$.
We have to show that $R'$ is a 
refinement of $S$.
Let $R$ be the set of maps 
$f\colon T\to S$ such that 
there is $x\in Ob(F_{T})$
with $u_{T}x\cong f^{\ast}v_{S}(y)$ 
in $H_{T}$. By assumption
$R$ is a refinement of $S$.
Since $v_{T}f^{\ast}(y)\cong 
f^{\ast}v_{S}(y)$ we have $R'\subset R$.
Conversely, let $f\colon T\to S$ be
in $R$ and $x$ as above. Let $\xi$ 
be the isomorphism 
$u_{T}(x)\cong v_{T}f^{\ast}(y)$.
By assumption there are 
$y'\in Ob(G_{T})$ and 
an isomorphism $y'\cong f^{\ast}(y)$
in $G_{T}$ which is sent by
$v_{T}$ to $\xi$. In particular
$u_{T}(x)=v_{T}(y')$ and so
$R\subset R'$.

\emph{Step 3.} 
Suppose that in the above
pullback diagram the map 
$u$ is bicovering. We 
can form the cube diagram
\begin{displaymath}
\xymatrix@=2ex{
{\mathsf{S}F\times_{\mathsf{S}H}\mathsf{S}G} 
\ar[rr]^-{(\mathsf{S}u)'} \ar[dr] \ar[dd]
&& {\mathsf{S}G} \ar[drr]^{\mathsf{S}v} \ar'[d][dd]\\
& {\mathsf{S}F} \ar[rrr]^{\mathsf{S}u} \ar[dd]
&&& {\mathsf{S}H} \ar[dd]\\
{F\times_{H}G} \ar'[r][rr] \ar[dr]
&& {G} \ar[drr]^{v}\\
& {F} \ar[rrr]^{u}
&&& {H}
}
\end{displaymath}
One clearly has $\mathsf{S}(F\times_{H}G)
\cong \mathsf{S}F\times_{\mathsf{S}H}\mathsf{S}G$.
By 2.3(5) the vertical arrows
of the cube diagram are trivial fibrations
and the map $\mathsf{S}v$ is
an isofibration. By Proposition 30(1) 
$\mathsf{S}u$ is bicovering, hence 
by Steps 1 and 2 the map 
$(\mathsf{S}u)'$ is bicovering, so 
$u'$ is bicovering.
\end{proof}
\begin{cor}
Let $F\in {\rm Fib}(E)$
and $u$ be a bicovering map.
Then $F\times u$ is a bicovering map.
\end{cor}
\begin{proof}
This follows from Example 19(1) 
and Lemma 33.
\end{proof}
The next result is the first part of 
\cite[Chapitre II Corollaire 2.1.5]{Gi},
with a different proof.
\begin{cor}
If $G$ is a stack then so is 
$\mathsf{CART}(F,G)$ for every 
$F\in {\rm Fib}(E)$.
\end{cor}
\begin{proof}
This follows from 2.3(4), Example 
32 and Corollary 34.
\end{proof}
\begin{lem}
An object of {\rm Fib}$(E)$ that has the right 
lifting property with respect to all maps that 
are both cofibrations and $\mathcal{C}$-local 
equivalences is a stack.
\end{lem}
\begin{proof}
Let $F$ be as in the statement of the Lemma. 
Using Theorem 13 we factorize the map  
$aF\colon F\to \mathsf{A}F$ 
as a cofibration $F\to G$ followed by a
trivial fibration $G\to {\rm A}F$.
By Proposition 30(2) $G$ is a stack. 
By hypothesis the diagram
\begin{displaymath}
\xymatrix{
{F}\ar @{=} [r] \ar[d]
&{F}\\
{G}\\
}
\end{displaymath}
has a diagonal filler, therefore $F$ 
is a retract of $G$. By Proposition 
30(4) $F$ is a stack. 
\end{proof}
\begin{proof}[Proof of Theorem ~\ref{second}]
We shall apply Lemma 9 to 
the natural generalized model category 
{\rm Fib}$(E)$ (Theorem 13). 
Since we have Proposition 30(1),
it only remains to prove that every map  
$F\to G$ of {\rm Fib}$(E)$ can be factorized 
as a map that is both a cofibration and a 
$\mathcal{C}$-local equivalence
followed by a map that has the right 
lifting property with respect to all maps that 
are both cofibrations and $\mathcal{C}$-local 
equivalences. Consider the diagram 
\begin{displaymath}
\xymatrix{
{F}\ar[r]^{aF} \ar[d]
&{\mathsf{A}F} \ar[d]\\
{G} \ar[r]^{aG}
&{\mathsf{A}G}
}
\end{displaymath}
We can factorize the map $\mathsf{A}F\to \mathsf{A}G$ 
as a map $\mathsf{A}F\to H$ that is an
$E$-equivalence followed by an isofibration
$H\to \mathsf{A}G$. By Proposition 30(2) $H$ is a stack, 
so by Lemma 31 the map $H\to \mathsf{A}G$  has the 
right lifting property with respect to all maps that are both 
cofibrations and $\mathcal{C}$-local equivalences. 
Therefore the pullback map $G\times_{\mathsf{A}G}H\to G$ 
has the right lifting property with respect to all maps that 
are both cofibrations and $\mathcal{C}$-local equivalences. 
By Lemma 33 the map $G\times_{\mathsf{A}G}H\to H$ 
is bicovering, therefore the canonical map 
$F\to G\times_{\mathsf{A}G}H$ is bicovering. We 
factorize it as a cofibration $F\to K$ followed by a trivial  
fibration $K\to G\times_{\mathsf{A}G}H$.
The desired factorization is $F\rightarrow K$ followed
by the composite $K\to G\times_{\mathsf{A}G}H\to G$.

The fact that the fibrant objects of {\rm Champ}$(E)$
are the stacks follows from Lemmas 31 and 36.
Left properness of {\rm Champ}$(E)$ is a 
consequence of the left properness of 
{\rm Fib}$(E)$ and right properness is a 
consequence of Lemma 33.
\end{proof}
\begin{prop}
Every fibration of {\rm Champ}$(E)$ is a 
$\mathcal{C}$-local fibration.
\end{prop}
\begin{proof}
Let $F\to G$ be a fibration of {\rm Champ}$(E)$.
The argument used in the proof of Theorem 29
shows that $F\to G$ is a retract of the composite 
$K\to G\times_{\mathsf{A}G}H\to G$. We 
conclude by Proposition 30((3) and (4)).
\end{proof}
\begin{prop}[Compatibility with the $2$-category structure]
Let $u\colon F\to G$ be a cofibration 
and $v\colon H\to K$ a fibration in {\rm Champ}$(E)$. 
Then the canonical map
\begin{displaymath}
\xymatrix{
{{\bf Cart}_{E}(G,H)} \ar[r] &{{\bf Cart}_{E}(G,K)
\times_{{\bf Cart}_{E}(F,K)} {\bf Cart}_{E}(F,H)}
}
\end{displaymath}
is an isofibration that is a surjective 
equivalence if either $u$ or $v$ is a
$\mathcal{C}$-local equivalence.
\end{prop}
\begin{proof}
The first part is contained in 
Proposition 21 since every fibration
of {\rm Champ}$(E)$ is an isofibration.
If $v$ is a $\mathcal{C}$-local equivalence
then $v$ is a trivial fibration and 
the Proposition is contained in
Proposition 21. Suppose that $u$ is 
a $\mathcal{C}$-local equivalence.
By adjunction it suffices to prove
that for every injective on 
objects functor $A\to B$, the
canonical map
\begin{displaymath}
\xymatrix{
{A\times G\sqcup_{A\times F}B\times F}
\ar[r] & {B\times G}
}
\end{displaymath}
is a cofibration and a 
$\mathcal{C}$-local 
equivalence (the pushout
in the displayed arrow 
exists by Lemma 17(2)).
This follows, for example, 
from Corollary 34.
\end{proof}
\begin{cor}[Compatibility with 
the `internal hom']
Let $u\colon F\to G$ be a cofibration 
and $v\colon H\to K$ a fibration in 
{\rm Champ}$(E)$. Then the 
canonical map
\begin{displaymath}
\xymatrix{
{\mathsf{CART}(G,H)}\ar[r] &{\mathsf{CART}(G,K)
\times_{\mathsf{CART}(F,K)} \mathsf{CART}(F,H)}
}
\end{displaymath}
is a trivial fibration
if either $u$ or $v$ is a 
$\mathcal{C}$-local equivalence.
\end{cor}
\begin{proof}
If $v$ is a $\mathcal{C}$-local equivalence
then $v$ is a trivial fibration and the Corollary
is Corollary 23. If $u$ is a $\mathcal{C}$-local 
equivalence the result follows from Proposition 38.
\end{proof}
\begin{prop}
The classes of bicoverings and isofibrations 
make {\rm Fib}$(E)$ a category of fibrant 
objects \cite{Br}.
\end{prop}
\begin{proof}
A path object was constructed in Remark 20. 
Since we have Lemma 17(1), we conclude by 
the next result.
\end{proof}
\begin{lem}
The maps that are both bicoverings 
and isofibrations are closed under pullbacks.
\end{lem}
\begin{proof}
A proof entirely similar to the 
proof of Lemma 33 can be given. 
We shall give a proof that uses
Lemma 33. Let 
\begin{displaymath}
\xymatrix{
{F\times_{H}G}\ar[r]^-{u'} 
\ar[d]&{G} \ar[d]^{v}\\
{F} \ar[r]^{u}
&{H}
}
\end{displaymath}
be a pullback diagram
in {\rm Fib}$(E)$ with
$u$ both an isofibration 
and a bicovering map.
We factorize $v$ as
$v=pj\colon G\to K\to H$,
where $j$ is an $E$-equivalence
and $p$ is an isofibration
and then we take 
successive pullbacks.
The map $F\times_{H}K\to K$
is a bicovering map by 
Lemma 33 and an isofibration.
The map $F\times_{H}G\to F\times_{H}K$
is an $E$-equivalence. By Proposition 
30(1) the map $u'$ is bicovering.
\end{proof}
We give now, as Lemma 42, 
the analogues, in our context, 
of \cite[Lemma 2.2 and Remark 2.3]{Ho2}.

Let $E$ be a category. Let
\begin{displaymath}
\xymatrix{
{F}\ar[r]^{u} & {H} & {G}\ar[l]_{v}
}
\end{displaymath}
be a digram in {\rm Fib}$(E)$.
The discussion from Section 3.4
and Remark 20 suggest the following 
model for the homotopy pullback
of the previous diagram. The objects 
of the fibre category over $S\in Ob(E)$
are triples $(x,y,\theta)$ with 
$x\in Ob(F_{S}),y\in Ob(G_{S})$ 
and $\theta \colon u(x)\to v(y)$ an 
isomorphism in $H_{S}$. The arrows 
are pairs of arrows making the obvious 
diagram commute. This model 
is commonly known as the $2$-pullback
or the iso-comma object of $u$ and $v$
and from now on we shall designate it by 
$F\times_{H}^{h}G$.
\begin{lem}[Homotopy pullbacks 
in {\rm Champ}$(E)$]
Suppose that $E$ is a site.

{\rm (1)} If $F,G$ and $H$ are stacks, 
then $F\times_{H}^{h}G$ is weakly
equivalent in {\rm Fib}$(E)$ to the
homotopy pullback in {\rm Champ}$(E)$
of the previous diagram.

{\rm (2)} $F\times_{H}^{h}G$ is 
weakly equivalent in {\rm Champ}$(E)$ 
to the homotopy pullback in {\rm Champ}$(E)$
of the previous diagram.
\end{lem}
\begin{proof}
(1) follows from Proposition 11(1). 
(2) follows from Proposition 11(2) 
and Lemma 41.
\end{proof}
Let $E$ and $E'$ be two sites and
$f\colon E\to E'$ be a category
fibred in groupoids. Then for every 
$S\in Ob(E)$, the induced map 
$f_{/S}\colon E_{/S}\to E'_{/f(S)}$
sends sieves to sieves.
Recall from Section
2.3 the adjoint pair $(f^{\bullet},
f_{\bullet}^{fib})$.
\begin{prop}[A change of base]
Let $E$ and $E'$ be two sites and
$f\colon E\to E'$ be a category
fibred in groupoids. Suppose that
for every $S\in Ob(E)$, the map 
$f_{/S}$ sends a refinement of $S$
to a refinement of $f(S)$. Then
the adjoint pair $(f^{\bullet},
f_{\bullet}^{fib})$ is a Quillen pair
between {\rm Champ}$(E)$
and {\rm Champ}$(E')$.
\end{prop}
\begin{proof}
Since we have Proposition 25,
it suffices to show that $f_{\bullet}^{fib}$
preserves stacks (Sections 3.5
and 3.3). Let $F$ be a stack in 
{\rm Fib}$(E')$, $S\in Ob(E)$ and
$R$ be a refinement of $S$. The map 
$f_{/S}$ is an $E'$-equivalence
and its restriction to $R$ is an
$E'$-equivalence $f_{/S}\colon R\to f_{/S}(R)$.
It follows from Proposition 14 that the
maps ${\bf Cart}_{E'}(f_{/S},F)$
are equivalences. We have the
following commutative diagram 
(see 2.3(7))
\begin{displaymath}
\xymatrix{
{{\bf Cart}_{E'}(E'_{/f(S)},F)}
\ar[r]\ar[d]&{{\bf Cart}_{E'}(E_{/S},F)}
\ar[d]\ar[r]^{\cong}&
{{\bf Cart}_{E}(E_{/S},f_{\bullet}^{fib}F)}\ar[d]\\
{{\bf Cart}_{E'}(f_{/S}(R),F)}\ar[r]
&{{\bf Cart}_{E'}(R,F)}\ar[r]^{\cong}&
{{\bf Cart}_{E}(R,f_{\bullet}^{fib}F)}
}
\end{displaymath}
The left vertical arrow is an equivalence
by assumption, hence $f_{\bullet}^{fib}F$
is a stack.
\end{proof}

\section{Categories fibred in groupoids}

Let $E$ be a category. In this section 
we give the analogues of Theorems 
13 and 29 for the category 
{\rm Fibg}$(E)$ defined in Section 2.3.
\begin{thm}
The category {\rm Fibg}$(E)$ is a proper 
generalized model category 
with the $E$-equivalences as weak equivalences,
the maps that are injective on objects as 
cofibrations and the isofibrations as fibrations.
\end{thm}
\begin{proof}
It only remains to check axiom A2 
from Definition 8. This is satisfied
by Remark 18.
\end{proof}
Suppose now that $E$ is a site. Notice
that for every $S\in Ob(E)$ and every 
refinement $R$ of $S$, $E_{/S}$ and $R$
are objects of {\rm Fibg}$(E)$.
We recall from \cite[Chapitre II D\'efinition 1.2.1(ii)]{Gi}
that an object $F$ of {\rm Fib}$(E)$ is a 
\emph{prestack} if for every $S\in Ob(E)$
and every refinement $R\subset E_{/S}$
of $S$, the restriction functor
\begin{displaymath}
{\bf Cart}_{E}(E_{/S},F)\to {\bf Cart}_{E}(R,F)
\end{displaymath}
is full and faithful.
\begin{lem}
{\rm (}\cite{Gi} {\rm and} 
\cite[Proposition 4.20]{Vi}{\rm )}
If an object $F$ of {\rm Fib}$(E)$ is 
a stack, so is $F^{cart}$. The converse 
holds provided that $F$ is a prestack.
\end{lem}
\begin{proof}
We recall that $max\colon CAT\to GRPD$ 
denotes the maximal groupoid functor.
We recall that an arbitrary functor $f$
is essentially surjective if and only if
the functor $max(f)$ is so and that 
if $f$ is full and faithful then so is 
$max(f)$. The Lemma follows then from 
the following commutative diagram
(see 2.3(6))
\begin{displaymath}
\xymatrix{
{{\bf Cartg}_{E}(E_{/S},F^{cart})}
\ar[r]^{\cong} \ar[d]&
{max{\bf Cart}_{E}(E_{/S},F)} \ar[d]\\
{{\bf Cartg}_{E}(R,F^{cart})} \ar[r]^{\cong}
&{max{\bf Cart}_{E}(R,F)}
}
\end{displaymath}
\end{proof}
Let $F$ be an object of {\rm Fibg}$(E)$,
$G$ an object of {\rm Fib}$(E)$ and 
$u\colon F\to G$ a bicovering map.
We claim that $u^{cart}$ is a bicovering
map as well. For, consider the diagram
\begin{displaymath}
\xymatrix{
{F^{cart}}\ar @{=}[r]
\ar[d]_{u^{cart}}&{F} \ar[d]^{u}\\
{G^{cart}} \ar[r]
&{G}
}
\end{displaymath}
One can readily check that the 
inclusion map $G^{cart}\to G$
is an isofibration and that by Lemma 
17(1) the above diagram is a pullback.
We conclude by Lemma 33. If, 
in addition, $G$ is a stack, then
the map $G^{cart}\to G$ is a 
bicovering map between stacks 
(see Lemma 45), hence by 
\cite[Chapitre II Proposition 1.4.5]{Gi}
it is an $E$-equivalence. It follows that 
$G$ is an object of {\rm Fibg}$(E)$.
\begin{thm}
There is a proper generalized model category 
{\rm Champg}$(E)$ on the category 
{\rm Fibg}$(E)$ in which the weak 
equivalences are the bicovering maps,
the cofibrations are the maps 
that are injective on objects and the
fibrantions are the fibrations of 
{\rm Champ}$(E)$.
\end{thm}
\begin{proof}
Using Theorem 44 and Lemma 9
it only remains to prove 
the factorization of an arbitrary 
map of {\rm Fibg}$(E)$ into a 
map that is both a cofibration and 
bicovering followed by a map 
that has the right lifting property 
with respect to all maps that are both 
cofibrations and bicoverings. The 
argument is the same as the one 
given in the proof of Theorem 29.
For it to work one needs the functor
$\mathsf{A}$ to send objects of 
{\rm Fibg}$(E)$ to objects of 
{\rm Fibg}$(E)$. This is so by the 
considerations preceding the 
statement of the Theorem, applied
to the map $F\to \mathsf{A}F$.
\end{proof}

\section{Sheaves of categories}

We begin by recalling the notion
of sheaf of categories.

Let $E$ be a small site. We recall
that $\widehat{E}$ is the category of 
presheaves on $E$ and $\eta\colon 
E\to \widehat{E}$ is the Yoneda 
embedding. We denote by $\widetilde{E}$ 
the category of sheaves on $E$ and 
by $a$ the associated sheaf functor, 
left adjoint to the inclusion functor 
$i\colon \widetilde{E}\to \widehat{E}$.

We denote by $\underline{Hom}$
the internal $CAT$-hom of the 
$2$-category $[E^{op},CAT]$
and by $X^{(A)}$ the cotensor 
of $X\in [E^{op},CAT]$ with 
a category $A$. Let 
$Ob\colon CAT\to SET$ denote
the set of objects functor;
it induces a functor $Ob\colon 
[E^{op},CAT]\to\widehat{E}$.
\begin{lem} 
Let $X$ be an object of 
$[E^{op},CAT]$. The 
following are equivalent.

$(a)$ For every category $A$, 
$ObX^{(A)}$ is a sheaf 
\cite[Expos\'e ii D\'{e}finition 6.1]{AGV}.

$(b)$ For every $S\in Ob(E)$ and
every refinement $R$ of $S$,
the natural map
\begin{displaymath}
\underline{Hom}(D\eta(S),X)\to
\underline{Hom}(DR',X)
\end{displaymath}
is an isomorphism, where $R'$ is
the sub-presheaf of $\eta(S)$ which 
corresponds to $R$.

$(c)$ For every $S\in Ob(E)$ and
every refinement $R$ of $S$,
the natural map
\begin{displaymath}
X(S)\to \underset{R^{op}}{\rm lim}(X|R)
\end{displaymath}
is an isomorphism, 
where $(X|R)$ is the composite
$R^{op}\to (E_{/S})^{op}\to E^{op}
\overset{X}\to CAT$.
\end{lem}
An object $X$ of $[E^{op},CAT]$
is a \emph{sheaf on} $E$ 
\emph{with values in} $CAT$ 
(simply, \emph{sheaf of categories})
if it satisfies one of the conditions of
Lemma 47. We denote by 
{\rm Faisc}$(E;CAT)$ the full subcategory 
of $[E^{op},CAT]$ whose objects 
are the sheaves of categories.
The category $[E^{op},CAT]$ is
equivalent to the category
{\rm Cat}$(\widehat{E})$ 
of internal categories and internal 
functors in $\widehat{E}$ and
{\rm Faisc}$(E;CAT)$ is equivalent to 
the category {\rm Cat}$(\widetilde{E})$
of internal categories and internal functors 
in $\widetilde{E}$ 
\cite[Expos\'e ii Proposition 6.3.1]{AGV}.

Consider now the adjunctions 
\begin{displaymath}
\xymatrix{
{{\rm Fib}(E)}\ar[r]^{\mathsf{L}}&
{{\rm Cat}(\widehat{E})}\ar@<1ex>[l]^{\Phi}
\ar[r]^{a}&{{\rm Cat}(\widetilde{E})}
\ar@<1ex>[l]^{i}
}
\end{displaymath}
(see Section 2.3 for the adjoint pair
$(\mathsf{L},\Phi)$). We denote the 
unit of the adjoint pair $(a,i)$ by $\mathsf{k}$. 
\begin{thm}
\label{third}
There is a right proper model category
{\rm Stack}$(\widetilde{E})_{proj}$
on the category {\rm Cat}$(\widetilde{E})$ 
in which the weak equivalences
and the fibrations are the maps that $\Phi$
takes into weak equivalences and fibrations
of {\rm Champ}$(E)$. The adjoint pair 
$(a\mathsf{L},\Phi i)$ is a Quillen equivalence 
between {\rm Champ}$(E)$ and 
{\rm Stack}$(\widetilde{E})_{proj}$.
\end{thm}
The prove the existence of 
the model category 
{\rm Stack}$(\widetilde{E})_{proj}$
we shall use  Lemma 49 below and 
the following facts:

(1) if $X$ is a sheaf of categories
then $\Phi X$ is a prestack 
\cite[Chapitre II 2.2.1]{Gi};

(2) for every $X\in [E^{op},CAT]$, the
natural map $\Phi \mathsf{k}(X)
\colon \Phi X\to \Phi iaX$ is bicovering
 \cite[Chapitre II Lemme 2.2.2(ii)]{Gi};

(3) if $X$ is a sheaf of categories
then $\Phi ia\mathcal{S}\Phi X$ is a 
stack (which is a consequence of) 
\cite[Chapitre II Lemme 2.2.2(iv)]{Gi}.

See the end of this section for 
another proof of (1).

We recall that the weak equivalences 
of {\rm Stack}$(\widetilde{E})_{proj}$ 
have a simplified description. Let $f$ be 
a map of {\rm Cat}$(\widetilde{E})$.
By \cite[Chapitre II Proposition 1.4.5]{Gi} 
the map $\Phi f$ is bicovering
if and only if $\Phi f$ is full and faithful 
and $\Phi f$ is `locally essentially 
surjective on objects'. Given any map $u$ 
of {\rm Fib}$(E)$, the underlying functor
of $u$ is full and faithful if and only 
if for every $S\in Ob(E)$, $u_{S}$
is full and faithful 
\cite[Expos\'e VI Proposition 6.10]{Gr}. 
Hence $f$ is a weak equivalence if and 
only if $f$ is full and faithful and $\Phi f$ 
is `locally essentially surjective on objects'.

\begin{lem}
Let $\mathcal{M}$ be a generalized 
model category. Suppose that there is a
set $I$ of maps of $\mathcal{M}$
such that a map of $\mathcal{M}$
is a trivial fibration if and only
if it has the right lifting property
with respect to every element of $I$.
Let $\mathcal{N}$ be a complete and 
cocomplete category and let 
$F\colon \mathcal{M}\rightleftarrows 
\mathcal{N}\colon G$ be a pair of adjoint
functors. Assume that

(1) the set $F(I)=\{F(u)\ |\ u\in I\}$ permits 
the small object argument 
\cite[Definition 10.5.15]{Hi};

(2) $\mathcal{M}$ is right proper;

(3) $\mathcal{N}$ has a fibrant
replacement functor, which means 
that there are 

$(i)$ a functor $\widehat{{\rm F}}\colon \mathcal{N}
\to \mathcal{N}$ such that for every object $X$ of 
$\mathcal{N}$ the object $G\widehat{{\rm F}}X$ 
is fibrant and

$(ii)$ a natural transformation from the identity functor 
of $\mathcal{N}$ to $\widehat{{\rm F}}$ such that 
for every object $X$ of $\mathcal{N}$ the map 
$GX\to G\widehat{{\rm F}}X$ is a
weak equivalence;

(4) every fibrant object of 
$\mathcal{N}$ has a path object, which 
means that for every object $X$ of 
$\mathcal{N}$ such that $GX$ is
fibrant there is a factorization 
\begin{displaymath}
\xymatrix{
{X}\ar[r]^{s} &{PathX} 
\ar[r]^{p_{0}\times p_{1}} & {X\times X}
}
\end{displaymath}
of the diagonal map $X\rightarrow X\times X$
such that $G(s)$ is a weak equivalence
and $G(p_{0}\times p_{1})$ is a fibration.

Then $\mathcal{N}$ becomes a right proper
model category in which the weak equivalences
and the fibrations are the maps that $G$
takes into weak equivalences and fibrations.

The adjoint pair $(F,G)$ is a Quillen
equivalence if and only if for every
cofibrant object $A$ of $\mathcal{M}$,
the unit map $A\to GFA$ of the adjunction
is a weak equivalence.
\end{lem}
\begin{proof}
Let $f$ be a map of $\mathcal{N}$.
We say that $f$ is a \emph{trivial fibration} 
if $G(f)$ is a trivial fibration and we say that
$f$ is a \emph{cofibration} if it is an
$F(I)$-cofibration in the sense of
\cite[Definition 10.5.2(2)]{Hi}.
By (1) and \cite[Corollary 10.5.23]{Hi}
every map of $\mathcal{N}$
can be factorized into a cofibration
followed by a trivial fibration and 
every cofibration has the left lifting 
property with respect to every trivial 
fibration. 

Let $f\colon X\to Y$ be a map of $\mathcal{N}$
such that $GX$ and $GY$ are fibrant. Then (4)
implies that we can construct the mapping 
path factorization of $f$ (see 
Section 3.0.6, for instance), that is,
$f$ can be factorized into a map $X\to Pf$ 
that is a weak equivalence followed by a 
map $Pf\to Y$ that is a fibration. 
Moreover, $GPf$ is fibrant.

We show that every map $f\colon X\to Y$
of $\mathcal{N}$ can be factorized 
into a map that is both a cofibration 
and a weak equivalence followed by 
a map that is a fibration. 
By (3) we have a commutative diagram
\begin{displaymath}
\xymatrix{
{X} \ar[r] \ar[d]_{f} &{\widehat{{\rm F}}X}
\ar[d]^{\widehat{{\rm F}}f}\\
{Y} \ar[r] &{\widehat{{\rm F}}Y}
}
\end{displaymath}
The map $\widehat{{\rm F}}f$ 
can be factorized into a map 
$\widehat{{\rm F}}X\to P\widehat{{\rm F}}f$
that is a weak equivalence followed by a map 
$P\widehat{{\rm F}}f\to \widehat{{\rm F}}Y$
that is a fibration.
Let $Z$ be the pullback of $P\widehat{{\rm F}}f
\to \widehat{{\rm F}}Y$ along $Y\to \widehat{{\rm F}}Y$.
By (2) the map $Z\to P\widehat{{\rm F}}f$ is 
a weak equivalence, therefore the canonical 
map $X\to Z$ is a weak equivalence.
We factorize $X\to Z$ into a map $X\to X'$ 
that is a cofibration followed by a map $X'\to Z$
that is a trivial fibration. The desired factorization
of $f$ is $X\to X'$ followed by the composite
$X'\to Z\to Y$.

We show that every commutative 
diagram in $\mathcal{N}$
\begin{displaymath}
\xymatrix{
{A} \ar[r] \ar[d]_{j} & {X} \ar[d]^{p}\\
{B} \ar[r] & {Y}
}
\end{displaymath}
where $j$ is both a cofibration and a 
weak equivalence and $p$ is a fibration 
has a diagonal filler. We shall construct a 
commutative diagram
\begin{displaymath}
\xymatrix{
{A} \ar[r]  \ar[d]_{j} & {X'} 
\ar[r] \ar[d]^{q} & {X} \ar[d]^{p}\\
{B} \ar[r] & {Y'} \ar[r] & {Y}
}
\end{displaymath}
with $q$ a trivial fibration. We factorize the 
map $B\to Y$ into a map $B\to Y'$
that is both a cofibration and a weak equivalence
followed by a map $Y'\to Y$ that is a fibration. 
Similarly, we factorize the canonical 
map $A\to Y'\times_{Y}X$ into a map $A\to X'$
that is both a cofibration and a weak equivalence
followed by a map $X'\to Y'\times_{Y}X$ 
that is a fibration. Let $q$ be the 
composite map $X'\to Y'$. Then $q$ is 
a trivial fibration.

The model category $\mathcal{N}$
is right proper since $\mathcal{M}$ 
is right proper.

Suppose that $(F,G)$ is a Quillen
equivalence. Let $A$ be a cofibrant 
object of $\mathcal{M}$. We can find
a weak equivalence $f\colon FA\to X$ 
with $X$ fibrant. The composite map
$A\to GFA\to GX$ is the adjunct 
of $f$, hence it is a weak
equivalence. Thus, $A\to GFA$ is a 
weak equivalence. Conversely,
let $A$ be a cofibrant object of
$\mathcal{M}$ and $X$ a fibrant
object of $\mathcal{N}$. If $FA\to X$
is a weak equivalence then its
adjunct is the composite
$A\to GFA\to GX$, which is a weak 
equivalence. If $f\colon A\to GX$
is a weak equivalence, then it factorizes
as $A\to GFA\overset{Gf'}\to GX$,
where $f'$ is the adjunct of $f$. 
Hence $Gf'$ is a weak equivalence,
which means that $f'$ is a weak equivalence.
\end{proof}
\begin{proof}[Proof of Theorem ~\ref{third}]
In Lemma 49 we take 
$\mathcal{M}={\rm Champ}(E)$,
$\mathcal{N}={\rm Cat}(\widetilde{E})$,
$F=a\mathsf{L}$, $G=\Phi i$ and
$I$ to be the set of maps
$\{f\times E_{/S}\}$ with $S\in Ob(E)$ and
$f\in \mathscr{M}$, where $\mathscr{M}$
is the set of functors such that a functor
is a surjective equivalence if and only if it 
has the right lifting property with respect
to every element of $\mathscr{M}$.

By Theorem 13 and Proposition 16(2)
a map of {\rm Fib}$(E)$ is a trivial fibration
if and only if it has the right lifting property
with respect to every element of $I$.

We shall now check the assumptions 
(1)-(4) of Lemma 49. (1) and (2) 
are clear. We check (3).
Let $X$ be a sheaf of categories.
We put $\widehat{F}X=ia\mathcal{S}\Phi X$
and the natural transformation from the identity 
functor of {\rm Cat}$(\widetilde{E})$ to 
$\widehat{F}$ to be the composite map
\begin{displaymath}
\xymatrix{
{X}\ar[r]&{\mathcal{S}\Phi X}
\ar[r]^{\mathsf{k}(\mathcal{S}\Phi X)}
&{ia\mathcal{S}\Phi X}
}
\end{displaymath}
Assumption (3) of Lemma 49 is fulfilled
by the facts (2) and (3) mentioned
right after the statement of Theorem 48.
We check (4). Let $X$ be a sheaf of 
categories such that $\Phi X$ is a stack.
Let $J$ be the groupoid with two objects
and one isomorphism between them. The 
diagonal $X\to X\times X$ factorizes as
\begin{displaymath}
\xymatrix{
{X}\ar[r]^{s} &{X^{(J)}} 
\ar[r]^{p_{0}\times p_{1}} & {X\times X}
}
\end{displaymath}
Since $\Phi$ preserves cotensors 
and the cotensor of a stack and a 
category is a stack (2.3(2)), (4)
follows from Remark 20.

We now prove that $(a\mathsf{L},\Phi i)$
is a Quillen equivalence. For this we 
use Lemma 49. For every object $F$ 
of {\rm Fib}$(E)$, the unit 
$F\to \Phi ia\mathsf{L}F$ of this 
adjoint pair is the composite 
\begin{displaymath}
\xymatrix{
{F}\ar[r]^{\mathsf{l}_{F}}&
{\Phi \mathsf{L}F}
\ar[r]^{\Phi \mathsf{k}(\mathsf{L}F)}&
{\Phi ia\mathsf{L}F}
}
\end{displaymath}
which is a bicovering map.
\end{proof}
\begin{thm}
The model category
{\rm Cat}$(\widehat{E})_{proj}$ 
admits a proper left Bousfield localization
{\rm Stack}$(\widehat{E})_{proj}$
in which the weak equivalences
and the fibrations are the maps that $\Phi$
takes into weak equivalences and fibrations
of {\rm Champ}$(E)$. The adjoint pair 
$(\mathsf{L},\Phi)$ is a Quillen equivalence 
between {\rm Champ}$(E)$ and 
{\rm Stack}$(\widehat{E})_{proj}$.
\end{thm}
\begin{proof}
The proof is similar to the proof
of Theorem 48, using the adjoint
pair $(\mathsf{L},\Phi)$ and 
the fibrant replacement functor
\begin{displaymath}
\xymatrix{
{X}\ar[r]&{\mathcal{S}\Phi X}
\ar[r]^{\mathcal{S}(a\Phi X)}
&{\mathcal{S}\mathsf{A}\Phi X}
}
\end{displaymath}
\end{proof}
\begin{prop}[Compatibility with 
the $2$-category structure]
Let $A\to B$ be an injective on
objects functor and $X\to Y$
a fibration of 
{\rm Stack}$(\widetilde{E})_{proj}$.
Then the canonical map
\begin{displaymath}
\xymatrix{
{X^{(B)}} \ar[r] &{X^{(A)}
\times_{Y^{(A)}}Y^{(B)}}
}
\end{displaymath}
is a fibration that is a trivial
fibration if $A\to B$ is an 
equivalence of categories 
or $X\to Y$ is a weak equivalence.
\end{prop}
\begin{proof}
Since $\Phi$ preserves cotensors,
the Proposition follows from 
Proposition 38.
\end{proof}
We recall \cite[Theorem 4]{JT}
that {\rm Cat}$(\widetilde{E})$
is a model category in which
the weak equivalences are the 
maps that $\Phi$ takes into
bicovering maps and the cofibrations
are the internal functors that are 
monomorphisms on objects. 
See Appendix 3 for another 
approach to this result. We denote 
this model category by 
{\rm Stack}$(\widetilde{E})_{inj}$.
\begin{prop}
The identity functors on
${\rm Cat}(\widetilde{E})$ form
a Quillen equivalence between
{\rm Stack}$(\widetilde{E})_{proj}$
and {\rm Stack}$(\widetilde{E})_{inj}$.
\end{prop}
\begin{proof}
We show that the identity functor
${\rm Stack}(\widetilde{E})_{proj}
\to{\rm Stack}(\widetilde{E})_{inj}$
preserves cofibrations. For that, 
it suffices to show that for every
object $F$ of {\rm Fib}$(E)$ and every
injective on objects functor $f$, the map
$a\mathsf{L}(f\times F)$ is a cofibration
of {\rm Stack}$(\widetilde{E})_{inj}$.
The map $\mathsf{L}(f\times F)$
is objectwise injective on objects
(see the proof of Proposition 24),
which translates in {\rm Cat}$(\widehat{E})$
as: $\mathsf{L}(f\times F)$ is an internal 
functor having the property that is a 
monomorphism on objects. But the 
associated sheaf functor $a$ is known 
to preserve this property.

Since the classes of weak equivalences
of the two model categories are the 
same, the result follows.
\end{proof}
Let $E'$ be another small site and 
$f^{-1}\colon E\to E'$ be the functor
underlying a morphisms of sites
$f\colon E'\to E$ \cite[Chapitre 0 
D\'{e}finition 3.3]{Gi}. The adjoint 
pair $f^{\ast}\colon \widetilde{E}
\rightleftarrows \widetilde{E'}\colon f_{\ast}$
induces an adjoint pair 
$f^{\ast}\colon {\rm Cat}(\widetilde{E})
\rightleftarrows {\rm Cat}(\widetilde{E'})
\colon f_{\ast}$.
\begin{prop}[Change of site]
The adjoint pair $(f^{\ast},f_{\ast})$
is a Quillen pair between 
{\rm Stack}$(\widetilde{E})_{proj}$
and {\rm Stack}$(\widetilde{E'})_{proj}$.
\end{prop}
\begin{proof}
Consider the diagram
\begin{displaymath}
\xymatrix{
{{\rm Fib}(E)}&{{\rm Fib}(E')}\ar[l]_{f_{\bullet}^{fib}}\\
{[E^{op},CAT]}\ar[u]^{\Phi}&{[E'^{op},CAT]}
\ar[u]_{\Phi'}\ar[l]^{f_{\ast}}\\
{{\rm Cat}(\widetilde{E})}\ar[u]^{i}
&{{\rm Cat}(\widetilde{E'})}\ar[l]^{f_{\ast}}
\ar[u]_{i'}
}
\end{displaymath}
where $f_{\bullet}^{fib}$ was defined 
in Section 2.3 and  $f_{\ast}\colon 
[E'^{op},CAT]\to [E^{op},CAT]$ is the 
functor obtained by composing with $f$.
It is easy to check that the functor 
$f_{\bullet}^{fib}$ preserves 
isofibrations and trivial fibrations.
By \cite[Chapitre II Proposition 3.1.1]{Gi}
it also preserves stacks. Since 
$f_{\bullet}^{fib}\Phi '=\Phi
f_{\ast}$, it follows that $f_{\ast}$
preserves trivial fibrations and the 
fibrations between fibrant objects.
\end{proof}
Let $p\colon C\to I$ be a fibred site 
\cite[Expos\'{e} vi 7.2.1]{AGV2}
and $\tilde{p}\colon \tilde{C}^{/I}\to I$
be the (bi)fibred topos associated to $p$
\cite[Expos\'{e} vi 7.2.6]{AGV2}.
Using the above considerations 
we obtain a bifibration 
${\rm Cat}(\tilde{C}^{/I})\to I$ whose 
fibres are isomorphic to 
${\rm Cat}(\widetilde{C_{i}})$, hence
by they are model categories. 
Moreover, by Proposition 53 the inverse 
and direct image functors are Quillen pairs.

An elementary example of a fibred site
is the Grothendieck construction associated
to the functor that sends a topological
space $X$ to the category $\mathcal{O}(X)$
whose objects are the open 
subsets of $X$ and whose arrows 
are the inclusions of subsets.
\begin{prop}
Let $E$ and $E'$ be two small sites 
and $f\colon E\to E'$ be a category
fibred in groupoids. Suppose that
for every $S\in Ob(E)$, the map 
$E_{/S}\to E'_{/f(S)}$ 
sends a refinement of $S$
to a refinement of $f(S)$. Then
$f$ induces a Quillen pair between 
{\rm Stack}$(\widetilde{E})_{proj}$
and {\rm Stack}$(\widetilde{E'})_{proj}$.
\end{prop}
\begin{proof}
The proof is similar to the
proof of Proposition 53.
Consider the solid arrow diagram
\begin{displaymath}
\xymatrix{
{{\rm Fib}(E)}&{{\rm Fib}(E')}\ar[l]_{f_{\bullet}^{fib}}\\
{[E^{op},CAT]}\ar@<.5ex>[r]^{f_{!}}
\ar[u]^{\Phi}\ar@<.5ex>[d]^{a}
&{[E'^{op},CAT]}\ar@<.5ex>[d]^{a'}\ar[u]_{\Phi '}
\ar@<.5ex>[l]^{f^{\ast}}\\
{{\rm Cat}(\widetilde{E})}\ar@<.5ex>[u]^{i}
&{{\rm Cat}(\widetilde{E'})}\ar@{.>}[l]^{f^{\ast}}
\ar@<.5ex>[u]^{i'}
}
\end{displaymath}
where $f_{!}$ is the left adjoint to
the functor $f^{\ast}$ obtained by 
composing with $f$. We claim 
that the composition with $f$ functor 
$f^{\ast}\colon \widehat{E'}\to\widehat{E}$
preserves sheaves. Let  $X$ be a sheaf
on $E'$. By Example 28 it suffices to
show that $\Phi Df^{\ast}X$ is a stack.
But $\Phi Df^{\ast}X=f_{\bullet}^{fib}
\Phi 'DX$, so $f^{\ast}X$ is a sheaf by 
Proposition 43. Therefore, $f^{\ast}$
induces a functor $f^{\ast}\colon
{\rm Cat}(\widetilde{E'})\to
{\rm Cat}(\widetilde{E})$.
Since $f^{\ast}i'=if^{\ast}$, a
formal argument implies that 
$a'f_{!}i$ is left adjoint to $f^{\ast}$.

The fact that $(a'f_{!}i,f^{\ast})$
is a Quillen pair follows from 
Proposition 43.
\end{proof}
Here is an application of 
Proposition 54. For every
$S\in Ob(E)$, the category
$E_{/S}$ has the induced topology
\cite[Chapitre 0 3.1.4]{Gi}.
A map $T\to S$ of $E$ induces 
a category fibred in groupoids
$E_{/T}\to E_{/S}$. The assumption
of Proposition 54 is satisfied. 
By \cite[Chapitre II Proposition 3.4.4]{Gi} 
we obtain a stack over $E$ whose 
fibres are model categories and 
such that the inverse and direct 
image functors are Quillen pairs.

\subsection*{Sheaves of categories 
are prestacks}
Let $E$ be a site. We recall that an object $F$
of {\rm Fib}$(E)$ is a prestack if for every
$S\in Ob(E)$ and every refinement 
$R\subset E_{/S}$ of $S$, the 
restriction functor
\begin{displaymath}
{\bf Cart}_{E}(E_{/S},F)\to {\bf Cart}_{E}(R,F)
\end{displaymath}
is full and faithful. 

We give here an essentially-from-the-definition 
proof of \cite[Chapitre II 2.2.1]{Gi},
namely that if $X\in [E^{op},CAT]$
is a sheaf of categories
then $\Phi X$ is a prestack. 

Let first $X\in [E^{op},CAT]$. 
Let $D\colon SET\to CAT$ be the
discrete category functor; it
induces a functor $D\colon 
\widehat{E}\to [E^{op},CAT]$.
Let $R'$ be the sub-presheaf of 
$\eta(S)$ which corresponds to $R$.
We have the following commutative 
diagram
\begin{displaymath}
\xymatrix{
{{\bf Cart}_{E}(E_{/S},\Phi X)}
\ar[r]\ar[dd]&{(\Phi X)_{S}=
\underline{Hom}(D\eta(S),X)}
\ar[d]^{(I)}\\
&{\underline{Hom}(DR',X)}\ar[d]^{(II)}\\
{{\bf Cart}_{E}(R,\Phi X)}\ar[r]
&{\underline{Hom}(DR',\mathcal{S}\Phi X)}
}
\end{displaymath}
The top horizontal 
arrow is a surjective equivalence (2.3(5)).
Since $(\Phi,\mathcal{S})$ is a 
$2$-adjunction, the bottom horizontal 
arrow is an isomorphism. We will show 
below that the map $(II)$ is full and 
faithful. If $X$ is now a sheaf of 
categories, then the map $(I)$
is an isomorphism by Lemma 47,
therefore in this case $\Phi X$ is a prestack.

Let $P\in \widehat{E}$. We denote by
$E_{/P}$ the category $\Phi DP$.
 Let $m\colon E_{/P}\to E$
be the canonical map. The natural functor
\begin{displaymath}
m^{\ast}\colon [E^{op},CAT]
\to [(E_{/P})^{op},CAT]
\end{displaymath}
has a left adjoint $m_{!}$
that is the left Kan extension
along $m^{op}$. Since 
$m^{op}$ is an opfibration,
$m_{!}$ has a simple description.
For example, let $A$ be a category
and let $cA\in [(E_{/P})^{op},CAT]$ 
be the constant object at $A$; then 
$m_{!}cA$ is the tensor between $A$
and $DP$ in the $2$-category
$[E^{op},CAT]$. It follows that 
for every $X\in [E^{op},CAT]$ 
we have an isomorphism
\begin{displaymath}
\underset{(E_{/P})^{op}}{\rm lim}m^{\ast}
X\cong \underline{Hom}(DP,X)
\end{displaymath}
The map
$X\to \mathcal{S}\Phi X$ is objectwise
both an equivalence of categories and
injective on objects, hence so is the map
$m^{\ast}X\to m^{\ast}\mathcal{S}\Phi X$.
Therefore the map
\begin{displaymath}
\underset{(E_{/P})^{op}}{\rm lim}m^{\ast}
X\to \underset{(E_{/P})^{op}}{\rm lim}m^{\ast}
\mathcal{S}\Phi X
\end{displaymath}
is both full and faithful 
and injective on objects.

\section{Appendix 1: Stacks vs. 
the homotopy sheaf condition}

Throughout this section $E$
is a site whose topology
is generated by a pretopology.

\subsection{}
We recall that the model category
$CAT$ is a simplicial model category.
The cotensor $A^{(K)}$ between
a category $A$ and a simplicial set
$K$ is constructed as follows.

Let {\bf S} be the category of 
simplicial sets. Let $cat\colon {\bf S}\to CAT$ 
be the fundamental category functor,
left adjoint to the nerve functor.
Let $(-)_{1}^{-1}\colon CAT\to GRPD$
be the free groupoid functor, left
adjoint to the inclusion functor. Then
\begin{displaymath}
A^{(K)}=[(cat K)_{1}^{-1},A]
\end{displaymath}
One has $A^{(\Delta[n])}=[J^{n},A]$,
where $J^{n}$ is the free groupoid on $[n]$.

Let {\bf X} be a cosimplicial
object in $CAT$. The total object
of {\bf X} \cite[Definition 18.6.3]{Hi}
is calculated as
\begin{displaymath}
{\rm Tot}{\bf X}=\underline{Hom}(J,{\bf X})
\end{displaymath}
where $\underline{Hom}$ is the $CAT$-hom 
of the $2$-category $[\Delta,CAT]$ and $J$ 
is the cosimplicial object in $CAT$
that $J^{n}$ defines. The category
$\underline{Hom}(J,{\bf X})$ has a 
simple description. For $n\geq 2$,
$J^{n}$ is constructed from $J^{1}$
by iterated pushouts, so by adjunction 
an object of $\underline{Hom}(J,{\bf X})$
is a pair $(x,f)$, where $x\in Ob({\bf X}^{0})$
and $f\colon d^{1}(x)\to d^{0}(x)$
is an isomorphism of ${\bf X}^{1}$ 
such that $s^{0}(f)$ is the identity on 
$x$ and $d^{1}(f)=d^{0}(f)d^{2}(f)$.
An arrow $(x,f)\to (y,g)$ is an arrow
$u\colon x\to y$ of ${\bf X}^{0}$
such that $d^{0}(u)f=gd^{1}(u)$.

If {\bf X} moreover a coaugmented cosimplicial
object in $CAT$ with coaugmentation 
${\bf X}^{-1}$, there is a natural map
\begin{displaymath}
{\bf X}^{-1}\to {\rm Tot}{\bf X}
\end{displaymath}
We recall \cite[Theorem 18.7.4(2)]{Hi} 
that if {\bf X} is Reey fibrant in
$[\Delta,CAT]$, then the natural map
\begin{displaymath}
{\rm Tot}{\bf X}\to {\rm holim}{\bf X}
\end{displaymath}
is an equivalence of categories.

\subsection{}
For each $S\in Ob(E)$ and each
covering family $\mathscr{S}=(S_{i}\to  S)_{i\in I}$ 
there is a simplicial object $E_{/\mathscr{S}}$
in {\rm Fib}$(E)$ given by
\begin{displaymath}
(E_{/\mathscr{S}})_{n}=
\underset{i_{0},...,i_{n}\in I^{n+1}}
\coprod E_{/S_{i_{0},...,i_{n}}}
\end{displaymath}
where $S_{i_{0},...,i_{n}}=
S_{i_{0}}\times_{S}...\times_{S}S_{i_{n}}$.
$E_{/\mathscr{S}}$ is augmented
with augmentation $E_{/S}$.
\begin{prop}
An object $F$ of {\rm Fib}$(E)$
is a stack if and only if for every 
$S\in Ob(E)$ and every covering family 
$\mathscr{S}=(S_{i}\to  S)$,
the natural map
\begin{displaymath}
{\bf Cart}_{E}(E_{/S},F)\to 
{\rm Tot}{\bf Cart}_{E}(E_{/\mathscr{S}},F)
\end{displaymath}
is an equivalence of categories.
\end{prop}
\begin{proof}
The proof consists of 
unraveling the definitions.
\end{proof}
Following \cite{Ho1}, we say that
an object $F$ of {\rm Fib}$(E)$ 
\emph{satisfies the homotopy sheaf
condition} if for every $S\in Ob(E)$ and 
every covering family $\mathscr{S}=(S_{i}\to  S)$,
the natural map
\begin{displaymath}
{\bf Cart}_{E}(E_{/S},F)\to 
{\rm holim}{\bf Cart}_{E}(E_{/\mathscr{S}},F)
\end{displaymath}
is an equivalence of categories.
\begin{prop}\cite[Theorem 1.1]{Ho1}
An object of {\rm Fib}$(E)$ 
satisfies the homotopy sheaf
condition if and only if it is a stack.
\end{prop}
\begin{proof}
This follows from Proposition 55
and Lemma 57.
\end{proof}
\begin{lem}
For every $S\in Ob(E)$, every
covering family $\mathscr{S}=(S_{i}\to  S)_{i\in I}$ 
and every object $F$ of {\rm Fib}$(E)$,
the cosimplicial object in $CAT$
${\bf Cart}_{E}(E_{/\mathscr{S}},F)$
is Reedy fibrant.
\end{lem}
\begin{proof}
By Proposition 21 the adjoint pair 
\begin{displaymath}
F^{(-)}\colon CAT\rightleftarrows 
{\rm Fib}(E)^{op}\colon {\bf Cart}_{E}(-,F)
\end{displaymath}
is a Quillen pair. Therefore, to prove
the Lemma it suffices to show that 
$E_{/\mathscr{S}}$ is Reedy cofibrant,
by which we mean that for every 
$[n]\in Ob(\Delta)$ the latching object
of $E_{/\mathscr{S}}$ at $[n]$,
denoted by $L_{n}E_{/\mathscr{S}}$, 
exists and the natural map 
$L_{n}E_{/\mathscr{S}}\to 
(E_{/\mathscr{S}})_{n}$
is injective on objects.
A way to prove this 
is by using Lemma 58.
\end{proof}

\subsection{Latching objects 
of simplicial objects}

In general, the following considerations 
may help deciding whether a 
simplicial object in a generalized
model category is Reedy cofibrant.

Let $\mathcal{M}$ be a category 
and {\bf X} a simplicial object in 
$\mathcal{M}$. We recall that the 
latching object of {\bf X} at 
$[n]\in Ob(\Delta)$ is
\begin{displaymath}
L_{n}{\bf X}=\underset{\partial([n]\downarrow 
\overleftarrow{\Delta})^{op}}{\rm colim}{\bf X}
\end{displaymath}
provided that the colimit exists.
Here $\overleftarrow{\Delta}$
is the subcategory of $\Delta$
consisting of the surjective maps
and $\partial([n]\downarrow 
\overleftarrow{\Delta})$ is the
full subcategory of $([n]\downarrow 
\overleftarrow{\Delta})$ containing 
all the objects except the identity 
map of $[n]$. Below we shall review 
the construction of $L_{n}{\bf X}$.

The category $([n]\downarrow 
\overleftarrow{\Delta})$ has the 
following description \cite[VII 1]{GJ}.
The identity map of $[n]$ is its initial
object. Any other object is of the form
$s^{i_{1}}...s^{i_{k}}\colon [n]\to [n-k]$,
where $s^{j}$ denotes a codegeneracy 
operator, $1\leq k\leq n$ and 
$0\leq i_{1}\leq...\leq i_{k}\leq n-1$.

For $n\geq 0$ we let $\underline{n}$ be 
the set $\{1,2,...,n\}$, with the convention 
that $\underline{0}$ is the empty set.
We denote by $\mathcal{P}(\underline{n})$ 
the power set of $\underline{n}$. 
$\mathcal{P}(\underline{n})$ is a 
partially ordered set. We set
$\mathcal{P}_{0}(\underline{n})=
\mathcal{P}(\underline{n})
\setminus \{\emptyset\}$ and
$\mathcal{P}_{1}(\underline{n})=
\mathcal{P}(\underline{n})
\setminus \{\underline{n}\}$.
There is an isomorphism
\begin{displaymath}
([n]\downarrow 
\overleftarrow{\Delta})\cong
\mathcal{P}(\underline{n})
\end{displaymath}
which sends the  
identity map of $[n]$ to 
$\emptyset$ and the object
$s^{i_{1}}...s^{i_{k}}\colon 
[n]\to [n-k]$ as above to 
$\{i_{1}+1,...,i_{k}+1\}$.
Under this isomorphism
the category $\partial([n]
\downarrow \overleftarrow{\Delta})$
corresponds to 
$\mathcal{P}_{0}(\underline{n})$,
therefore $\partial([n]\downarrow 
\overleftarrow{\Delta})^{op}$ is isomorphic 
to $\mathcal{P}_{1}(\underline{n})$.
The displayed isomorphism is natural
in the following sense. Let $Dec^{1}\colon
\Delta \to\Delta$ be $Dec^{1}([n])=
[n]\sqcup [0]\cong [n+1]$.
Then we have a commutative diagram
\begin{displaymath}
\xymatrix{
{([n]\downarrow \overleftarrow{\Delta})}
\ar[r]^{\cong}\ar[d]_{Dec^{1}}
&{\mathcal{P}(\underline{n})}\ar[d]\\
{([n+1]\downarrow \overleftarrow{\Delta})}
\ar[r]^{\cong}&{\mathcal{P}(\underline{n+1})}
}
\end{displaymath}
in which the unlabelled vertical arrow 
is the inclusion. Restricting the arrow
$Dec^{1}$ to $\partial$ and then 
taking the opposite category we
obtain a commutative diagram
in which the unlabelled vertical
arrow becomes $-\cup \{n+1\}\colon
\mathcal{P}_{1}(\underline{n})
\to \mathcal{P}_{1}(\underline{n+1})$.

For $n\geq 1$ the category
$\mathcal{P}_{1}(\underline{n})$
is constructed inductively as
the Grothendieck
construction applied to the 
functor $(2\leftarrow 1\to 0)\to CAT$
given by $\ast\leftarrow 
\mathcal{P}_{1}(\underline{n-1})=
\mathcal{P}_{1}(\underline{n-1})$.
Therefore colimits indexed by 
$\mathcal{P}_{1}(\underline{n})$
have the following description.
Let ${\bf Y}\colon 
\mathcal{P}_{1}(\underline{n})
\to \mathcal{M}$. We denote by 
{\bf Y} the precomposition of 
{\bf Y} with the inclusion
$\mathcal{P}_{1}(\underline{n-1})
\subset \mathcal{P}_{1}(\underline{n})$;
then $\underset{\mathcal{P}_{1}(\underline{n})}
{\rm colim}{\bf Y}$ is the pushout of the diagram
\begin{displaymath}
\xymatrix{
{\underset{\mathcal{P}_{1}
(\underline{n-1})}{\rm colim}{\bf Y}}
\ar[r]\ar[d] &{\underset{\mathcal{P}_{1}
(\underline{n-1})}{\rm colim}{\bf Y}(-\cup \{n\})}\\
{{\bf Y}_{\underline{n-1}}}
}
\end{displaymath}
provided that the pushout and all
the involved colimits exist. 

Let $\overleftarrow{{\bf X}}$ 
be the restriction of {\bf X} to 
$(\overleftarrow{\Delta})^{op}$.
Notice that the definition of the 
latching object of {\bf X} uses 
only $\overleftarrow{{\bf X}}$.
Summing up, $L_{n}{\bf X}$ is 
the pushout of the diagram
\begin{displaymath}
\xymatrix{
{L_{n-1}\overleftarrow{{\bf X}}}
\ar[r]\ar[d] &{L_{n-1}Dec^{1}
(\overleftarrow{{\bf X}})}\\
{{\bf X}_{n-1}}
}
\end{displaymath}
provided that the pushout and all
the involved colimits exist, where
$\overleftarrow{{\bf X}}\to
Dec^{1}(\overleftarrow{{\bf X}})$ 
is induced by $s^{n}\colon 
[n]\sqcup [0]\to [n]$. Thus, we have
\begin{lem}
Let $\mathcal{M}$ be a generalized 
model category and {\bf X} a 
simplicial object in $\mathcal{M}$.
Let $n\geq 1$. If $L_{n-1}\overleftarrow{{\bf X}}$
and $L_{n-1}Dec^{1}(\overleftarrow{{\bf X}})$
exist and the map $L_{n-1}\overleftarrow{{\bf X}}
\to {\bf X}_{n-1}$ is a cofibration, then
$L_{n}{\bf X}$ exists.
\end{lem}
\section{Appendix 2: Left Bousfield localizations
and change of cofibrations}
In this section we essentially propose 
an approach to the existence
of left Bousfield localizations of 
`injective'-like model categories.
The approach is based on the existence 
of both the un-localized `injective'-like 
model category and the left Bousfield 
localization of the `projective'-like 
model category. We give a full description
of the fibrations of these localized 
`injective'-like model categories;
depending on one's  taste, the description
may or may not be satisfactory.
The approach uses only simple 
factorization and lifting arguments.

Let $\mathcal{M}_{1}=({\rm W},
{\rm C}_{1},{\rm F}_{1})$
and $\mathcal{M}_{2}=({\rm W},
{\rm C}_{2},{\rm F}_{2})$ 
be two model categories
on a category $\mathcal{M}$, 
where, as usual, {\rm W} stands
for the class of weak equivalences,
{\rm C} stands for the class of cofibrations, 
and {\rm F} for the class of fibrations.
We assume that ${\rm C}_{1}\subset {\rm C}_{2}$.
Let ${\rm W}'$ be a class of maps of 
$\mathcal{M}$ that contains {\rm W} 
and has the two out of three property. 
We define ${\rm F}_{1}'$
to be the class of maps having the 
right lifting property with respect
to every map of ${\rm C}_{1}\cap {\rm W}'$,
and we define ${\rm F}_{2}'$
to be the class of maps having the 
right lifting property with respect
to every map of 
${\rm C}_{2}\cap {\rm W}'$.
One can think of $\mathcal{M}_{1}$ 
as the `projective' model category, of 
$\mathcal{M}_{2}$ as the `injective' 
model category, and of ${\rm W}'$ as  
the class of `local', or `stable, equivalences'. 
Of course, other adjectives can be used.
Recall from Section 3.3 the notion of 
left Bousfield localization of a 
(generalized) model category.
\begin{thm}
(1) {\rm (Restriction)} If ${\rm L}\mathcal{M}_{2}=
({\rm W}',{\rm C}_{2},{\rm F}_{2}')$ 
is a left Bousfield localization of 
$\mathcal{M}_{2}$, then 
the class of fibrations of ${\rm L}\mathcal{M}_{2}$
is the class ${\rm F}_{2}\cap {\rm F}_{1}'$
and ${\rm L}\mathcal{M}_{1}=
({\rm W}',{\rm C}_{1},{\rm F}_{1}')$
is a left Bousfield localization of 
$\mathcal{M}_{1}$.

(2) {\rm (Extension)} If ${\rm L}\mathcal{M}_{1}
=({\rm W}',{\rm C}_{1},{\rm F}_{1}')$ 
is a left Bousfield localization 
of $\mathcal{M}_{1}$ that is right proper, 
then ${\rm L}\mathcal{M}_{2}=
({\rm W}',{\rm C}_{2},{\rm F}_{2}')$
is a left Bousfield localization of 
$\mathcal{M}_{2}$.
\end{thm}
For future purposes we display
the conclusion of Theorem 59(2) 
in the diagram
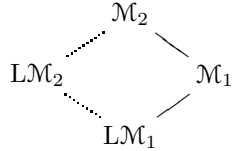
\begin{displaymath}
\xymatrix@=2ex{
&{\mathcal{M}_{2}}\ar@{-}[dr]\\
{{\rm L}\mathcal{M}_{2}}\ar@{.}[ur]\ar@{.}[dr]
&&{\mathcal{M}_{1}}\\
&{{\rm L}\mathcal{M}_{1}}\ar@{-}[ur]
}
\end{displaymath}
The proofs of the existence of the left 
Bousfield localizations in parts (1) and (2) 
are different from one another.
As it will be explained below, the existence 
of the left Bousfield localization in part (1) is 
actually well-known, but perhaps it has not 
been formulated in this form. Also, the right
properness assumption in part (2) is
dictated by the method of proof.
\begin{proof}[Proof of Theorem 59]
We prove part (1). We first show that 
${\rm F}_{2}'={\rm F}_{2}\cap {\rm F}_{1}'$.
Clearly, we have ${\rm F}_{2}'\subset 
{\rm F}_{2}\cap {\rm F}_{1}'$.
Conversely, we must prove that 
every commutative 
diagram in $\mathcal{M}$
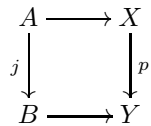
\begin{displaymath}
\xymatrix{
{A}\ar[r] 
\ar[d]_{j}&{X}\ar[d]^{p}\\
{B} \ar[r]
&{Y}
}
\end{displaymath}
where $j$ is in ${\rm C}_{2}\cap 
{\rm W}'$ and $p$ is in 
${\rm F}_{2}\cap {\rm F}_{1}'$,
has a diagonal filler. The idea, 
which we shall use again, is 
very roughly that a commutative 
diagram 
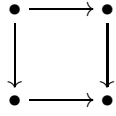
\begin{displaymath}
\xymatrix{
{\bullet}\ar[r] 
\ar[d]&{\bullet}\ar[d]\\
{\bullet} \ar[r]
&{\bullet}
}
\end{displaymath}
in an arbitrary category
has a diagonal filler when,
for example, viewed as 
an arrow going from left to 
right in the category of arrows,
it factors through an isomorphism.

We first construct a 
commutative diagram
\begin{displaymath}
\xymatrix{
{A} \ar[r]  \ar[d]_{j} & {X'} 
\ar[r] \ar[d]^{q} & {X} \ar[d]^{p}\\
{B} \ar[r] & {Y'} \ar[r] & {Y}
}
\end{displaymath}
with $q$ in ${\rm F}_{2}\cap {\rm W}$.
Then, since $j$ is in ${\rm C}_{2}$,
the left commutative square diagram
has a diagonal filler.

We factorize the map $B\to Y$ 
into a map $B\to Y'$ in 
${\rm C}_{2}\cap {\rm W}'$
followed by a map $Y'\to Y$ in
${\rm F}_{2}'$. We factorize
the canonical map $A\to Y'\times_{Y}X$ 
into a map $A\to X'$ in 
${\rm C}_{2}\cap {\rm W}'$
followed by a map $X'\to Y'\times_{Y}X$ 
in ${\rm F}_{2}'$. Let $q$ be the 
composite map $X'\to Y'$; then $q$
is in ${\rm F}_{2}$ being the 
composite of two maps
in ${\rm F}_{2}$.
On the other hand, $q$ is in 
${\rm F}_{1}'$ since 
${\rm F}_{2}'\subset {\rm F}_{1}'$
and since ${\rm F}_{1}'$ is stable 
under pullbacks and compositions.
By the two out of three 
property  $q$ is in ${\rm W}'$,
therefore $q$ belongs to 
${\rm F}_{1}'\cap {\rm W}'=
{\rm F}_{1}\cap {\rm W}$.
In all, $q$ is in ${\rm F}_{2}\cap {\rm W}$.

We now prove the existence of  
${\rm L}\mathcal{M}_{1}$.
This can be seen as a consequence of a 
result of M. Cole \cite[Theorem 2.1]{Co} (or of
B.A. Blander \cite[Proof of Theorem 1.5]{Bl}).
In our context however, since we have 
Lemma 9 we only need to check the 
factorization of an arbitrary map of 
$\mathcal{M}$ into a map
in ${\rm C}_{1}\cap {\rm W}'$ 
followed by a map in ${\rm F}_{1}'$.
This proceeds as in \cite{Co, Bl}; 
for completeness we reproduce the 
argument. 

Let $f\colon X\to Y$ be a map of $\mathcal{M}$.
We factorize it as a map $X\to Z$ in 
${\rm C}_{2}\cap {\rm W}'$ followed
by a map $Z\to Y$ in ${\rm F}_{2}'$.
We further factorize $X\to Z$ into a
map $X\to Z'$ in ${\rm C}_{1}$ followed
by a map $Z'\to Z$ in ${\rm F}_{1}\cap {\rm W}$.
The desired factorization of $f$ is $X\to Z'$
followed by the composite $Z'\to Y$.

We prove part (2). By Lemma 9, 
it only remains to check the 
factorization of an arbitrary map
of $\mathcal{M}$ into a map
in ${\rm C}_{2}\cap {\rm W}'$ 
followed by a map in ${\rm F}_{2}'$.
Mimicking the argument given in 
part (1) for the existence of 
${\rm L}\mathcal{M}_{1}$
does not seem to give a solution.
We shall instead expand on 
an argument due to A.K. Bousfield 
\cite[Proof of Theorem 9.3]{Bou},
that's why we assumed right properness
of ${\rm L}\mathcal{M}_{1}$.

\emph{Step 1.} We give an example 
of a map in ${\rm F}_{2}'$.
We claim that every commutative 
diagram in $\mathcal{M}$
\begin{displaymath}
\xymatrix{
{A}\ar[r] 
\ar[d]_{j}&{X}\ar[d]^{p}\\
{B} \ar[r]
&{Y}
}
\end{displaymath}
where $j$ is in ${\rm C}_{2}\cap 
{\rm W}'$, $p$ is in ${\rm F}_{2}$,
and $X$ and $Y$ are fibrant in 
${\rm L}\mathcal{M}_{1}$,
has a diagonal filler. For this 
we shall construct a 
commutative diagram
\begin{displaymath}
\xymatrix{
{A} \ar[r]  \ar[d]_{j} & {X'} 
\ar[r] \ar[d]^{q} & {X} \ar[d]^{p}\\
{B} \ar[r] & {Y'} \ar[r] & {Y}
}
\end{displaymath}
with $q$ in {\rm W}. Factorizing
then $q$ as a map in ${\rm C}_{2}$
followed by a map in 
${\rm F}_{2}\cap {\rm W}$ and
using two diagonal fillers, we obtain
the desired diagonal filler.
We factorize the map $B\to Y$ 
into a map $B\to Y'$ in 
${\rm C}_{1}\cap {\rm W}'$
followed by a map $Y'\to Y$ in
${\rm F}_{1}'$. We factorize
the canonical map $A\to Y'\times_{Y}X$ 
into a map $A\to X'$ in 
${\rm C}_{1}\cap {\rm W}'$
followed by a map $X'\to Y'\times_{Y}X$ 
in ${\rm F}_{1}'$. Let $q$ be the 
composite map $X'\to Y'$. By the two
out of three property $q$ is in 
${\rm W}'$. Since $Y$ is fibrant
in ${\rm L}\mathcal{M}_{1}$,
so is $Y'$. The map $Y'\times_{Y}X\to X$
is in ${\rm F}_{1}'$ and $X$
is fibrant in ${\rm L}\mathcal{M}_{1}$,
therefore  $Y'\times_{Y}X$, and hence
$X'$, are fibrant in 
${\rm L}\mathcal{M}_{1}$.
It follows that the map $q$ is in 
{\rm W}. The claim is proved.

\emph{Step 2.} Let $f\colon X\to Y$ 
be a map of $\mathcal{M}$. 
We can find a commutative diagram
\begin{displaymath}
\xymatrix{
{X}\ar[r] 
\ar[d]_{f}&{X'}\ar[d]^{f'}\\
{Y} \ar[r]
&{Y'}
}
\end{displaymath}
in which the two horizontal
arrows are in ${\rm W}'$
and both $X'$ and $Y'$ are
fibrant in ${\rm L}\mathcal{M}_{1}$.
We can find a commutative diagram
\begin{displaymath}
\xymatrix{
{X'}\ar[r] 
\ar[d]_{f'}&{X''}\ar[d]^{g}\\
{Y'} \ar[r]
&{Y''}
}
\end{displaymath}
in which the two horizontal
arrows are in {\rm W}, $g$
is in ${\rm F}_{2}$,
and both $X''$ and $Y''$ are
fibrant in $\mathcal{M}_{1}$.
It follows that both $X''$ and $Y''$ 
are fibrant in ${\rm L}\mathcal{M}_{1}$.
The map $g$ is a fibration in 
${\rm L}\mathcal{M}_{1}$,
since ${\rm F}_{2}\subset {\rm F}_{1}$.
Putting the two previous 
commutative diagrams side 
by side we obtain a commutative
diagram
\begin{displaymath}
\xymatrix{
{X}\ar[r] 
\ar[d]_{f}&{X''}\ar[d]^{g}\\
{Y} \ar[r]
&{Y''}
}
\end{displaymath}
in which the two horizontal
arrows are in ${\rm W}'$.
Since ${\rm L}\mathcal{M}_{1}$
is right proper, the map
$Y\times_{Y''}X''\to X''$ is
in ${\rm W}'$, therefore
the canonical map 
$X\to Y\times_{Y''}X''$
is in ${\rm W}'$. By the claim,
the map $Y\times_{Y''}X''\to Y$ 
is in ${\rm F}_{2}'$. We factorize
the map $X\to Y\times_{Y''}X''$
into a map $X\to Z$ that is in 
${\rm C}_{2}$ followed by a 
map $Z\to Y\times_{Y''}X''$
that is in ${\rm F}_{2}\cap {\rm W}$.
Since ${\rm F}_{2}\cap {\rm W}\subset
{\rm F}_{2}'$, we obtain the desired
factorization of $f$ into a map
in ${\rm C}_{2}\cap {\rm W}'$ 
followed by a map in ${\rm F}_{2}'$.
The proof of the existence of 
${\rm L}\mathcal{M}_{2}$ 
is complete.
\end{proof}

Some results in the subject of
`homotopical sheaf theory' can be 
seen as consequences of Theorem 59. 
Here are a couple of examples.

Let $\mathcal{C}$ be a small category. 
The category of presheaves on 
$\mathcal{C}$ with values in simplicial 
sets is a model category in two standard 
ways: it has the so-called projective 
and injective model structures. 
The class of cofibrations of the 
projective model category is 
contained in the class of cofibrations 
of the injective model category. If 
$\mathcal{C}$ is moreover a site, a 
result of Dugger-Hollander-Isaksen 
\cite[Theorem 6.2]{DHI} says that 
the projective model category
admits a left Bousfield localization
$U\mathcal{C}_{\mathcal{L}}$
at the class $\mathcal{L}$ of local
weak equivalences. The fibrations
of $U\mathcal{C}_{\mathcal{L}}$
are the objectwise fibrations that
satisfy descent for hypercovers
\cite[Theorem 7.4]{DHI}. The 
model category 
$U\mathcal{C}_{\mathcal{L}}$
is right proper (for an interesting
proof, see \cite[Proposition 7.1]{DI}).
Therefore, by Theorem 59(2), Jardine's
model category, denoted by 
$sPre(\mathcal{C})_{\mathcal{L}}$
in \cite{DHI}, exists. Moreover, 
by Theorem 59(1) its fibrations are 
the injective fibrations that satisfy 
descent for hypercovers: this is 
exactly the content of the first 
part of \cite[Theorem 7.4]{DHI}.
As suggested in \cite{DI}, this approach 
to $sPre(\mathcal{C})_{\mathcal{L}}$
reduces the occurence
of stalks and Boolean localization
technique. The category of presheaves on 
$\mathcal{C}$ with values in simplicial 
sets also admits the so-called 
flasque model category \cite[Theorem 3.7(a)]{Is}.
The class of cofibrations of the projective 
model category is contained in the class of 
cofibrations of the flasque model
category \cite[Lemma 3.8]{Is}. Using 
$U\mathcal{C}_{\mathcal{L}}$ and 
Theorem 59 it follows that the local flasque 
model category \cite[Definition 4.1]{Is} 
exists.

Other examples can be found
on page 199 of \cite{HSS}: the existence
of both therein called the $S$ model and 
the injective stable model structures, 
together with the description of their 
fibrations, can be seen as consequences 
of Theorem 59.

\section{Appendix 3: Strong stacks of 
categories revisited}
Let $E$ be a small site. Recall
from Theorem 48 the model category
$Stack(\widetilde{E})_{proj}$.
\begin{lem}
The class of weak equivalences of
$Stack(\widetilde{E})_{proj}$ is 
accessible.
\end{lem}
\begin{proof}
Let $f\colon X\to Y$ be a map of
${\rm Cat}(\widetilde{E})$. Consider
the commutative square diagram
\begin{displaymath}
\xymatrix{
{X}\ar[r] \ar[d]_{f}
&{\widehat{F}X} 
\ar[d]^{\widehat{F}f}\\
{Y} \ar[r]
&{\widehat{F}Y}
}
\end{displaymath}
with $\widehat{F}X$ defined
in the proof of Theorem 48.
Then $f$ is a weak equivalence
if and only if $i\widehat{F}f$
is a weak equivalence in 
$[E^{op},CAT]_{proj}$.
The functors $i$ and $\widehat{F}$
preserve $\kappa$-filtered colimits 
for some regular cardinal $\kappa$.
Since the class of weak equivalences of 
$[E^{op},CAT]_{proj}$ is accessible,
the result follows.
\end{proof}
Recall from Theorem 50 the (right
proper) model category
$Stack(\widehat{E})_{proj}$.
By Theorem 59(2)
we have the model category
$Stack(\widehat{E})_{inj}$,
which we display in the diagram
\begin{displaymath}
\xymatrix@=2ex{
&{[E^{op},CAT]_{inj}}\ar@{-}[dr]\\
{Stack(\widehat{E})_{inj}}\ar@{.}[ur]\ar@{.}[dr]
&&{[E^{op},CAT]_{proj}}\\
&{Stack(\widehat{E})_{proj}}\ar@{-}[ur]
}
\end{displaymath}
By Theorem 59(1), an
object $X$ of $[E^{op},CAT]$
is fibrant in $Stack(\widehat{E})_{inj}$ 
if and only if $\Phi X$ is a 
stack and $X$ is fibrant in 
$[E^{op},CAT]_{inj}$.
\begin{thm}
\cite[Theorem 4]{JT}
There is a model category
{\rm Stack}$(\widetilde{E})_{inj}$
on the category {\rm Cat}$(\widetilde{E})$
in which the weak equivalences are the 
maps that $\Phi$ takes into
bicovering maps and the cofibrations
are the internal functors that are 
monomorphisms on objects.
A sheaf of categories $X$ is 
fibrant in {\rm Stack}$(\widetilde{E})_{inj}$
(aka $X$ is a strong stack) if
and only if $\Phi X$ is a 
stack and $X$ is fibrant in 
$[E^{op},CAT]_{inj}$.
\end{thm}
\begin{proof}
We shall use J. Smith's recognition
principle for model categories
\cite[Theorem 1.7]{Be}. We
take in \emph{op. cit.} the class
{\rm W} to be the class of weak 
equivalences of $Stack(\widetilde{E})_{proj}$.
By Lema 60, {\rm W} is accessible.
Let $I_{0}$ be a generating set for
the class {\rm C} of cofibrations of 
$[E^{op},CAT]_{inj}$, so that 
${\rm C}={\rm cof}(I_{0})$.
We put $I=aI_{0}$. The functors $a$ 
and $i$ preserve the property of internal 
functors of being a monomorphism on 
objects. Using that $i$ is full and faithful
it follows that $a{\rm C}$ is the class 
of internal functors that are 
monomorphisms on objects, and that
moreover $a{\rm C}={\rm cof}(I)$.
By adjunction, every map in
${\rm inj}(I)$ is objectwise an 
equivalence of categories, so
in particular every such map is in 
{\rm W}. Recall that for every object 
$X$ of $[E^{op},CAT]$, the
natural map $X\to iaX$ is a 
weak equivalence (see fact (2)
stated below Theorem 48). 
Thus, by Lemma 62 all the 
assumptions of Smith's Theorem
are satisfied, so {\rm Cat}$(\widetilde{E})$
is a model category, which we denote by
{\rm Stack}$(\widetilde{E})_{inj}$.

Let $f$ be a fibration in this model
category. Then clearly $if$ is a 
fibration in $Stack(\widehat{E})_{inj}$.
Conversely, if $if$ is a fibration in
$Stack(\widehat{E})_{inj}$, then,
since $i$ is full and faithful,
$f$ is a fibration in 
{\rm Stack}$(\widetilde{E})_{inj}$.
\end{proof}
\begin{lem}
Let $\mathcal{M}$ be a model category. 
We denote by {\rm C} the class of 
cofibrations of $\mathcal{M}$.
Let $\mathcal{N}$ be a category and let 
$R\colon \mathcal{M}\rightleftarrows 
\mathcal{N}\colon K$ be a pair of adjoint
functors with $K$ full and faithful. 
We denote by {\rm W} the class of maps
of $\mathcal{N}$ that $K$ takes into
weak equivalences. Assume that

(1) $KR{\rm C}\subset {\rm C}$ and

(2) for every object $X$ of $\mathcal{M}$,
the unit map $X\to KRX$ is a weak equivalence. 

Then the class ${\rm W}\cap R{\rm C}$
is stable under pushouts and transfinite 
compositions.
\end{lem}
\begin{proof}
We first remark that by (2), the functor 
$R$ takes a weak equivalence to an 
element of {\rm W}. Let 
\begin{displaymath}
\xymatrix{
{X}\ar[r]^{f} \ar[d]
&{Y}\ar[d]\\
{Z} \ar[r]_{g}
&{P}
}
\end{displaymath}
be a pushout diagram
in $\mathcal{N}$ with
$f\in {\rm W}\cap R{\rm C}$.
Then $g$ is obtained by applying
$R$ to the pushout diagram
\begin{displaymath}
\xymatrix{
{KX}\ar[r]^{Kf} \ar[d]
&{KY}\ar[d]\\
{KZ} \ar[r]&{P'}
}
\end{displaymath}
By the assumptions it follows 
that $g\in {\rm W}\cap R{\rm C}$.
The case of transfinite compositions
is dealt with similarly.
\end{proof}


\begin{thebibliography}{18}
\bibitem{AGV}
M. Artin, A. Grothendieck, J. L. Verdier,
{\it Th\'eorie des topos et cohomologie \'etale 
des sch\'emas. Tome 1: Th\'eorie des topos},
S\'eminaire de G\'eom\'etrie Alg\'ebrique du 
Bois-Marie 1963--1964 (SGA 4),
Lecture Notes in Mathematics, Vol. 269. 
Springer-Verlag, Berlin-New York, 1972. xix+525 pp. 

\bibitem{AGV2}
M. Artin, A. Grothendieck, J. L. Verdier,
{\it Th\'eorie des topos et cohomologie \'etale 
des sch\'emas. Tome 2},
S\'eminaire de G\'eom\'etrie Alg\'ebrique du 
Bois-Marie 1963--1964 (SGA 4),
Lecture Notes in Mathematics, Vol. 270.
Springer-Verlag, Berlin-New York, 1972. iv+418 pp.

\bibitem{Be}
T. Beke, {\it Sheafifiable homotopy model categories} Math. Proc.
Cambridge Philos. Soc. 129 (2000), no. 3, 447--475.

\bibitem{Bl}
B. A. Blander, {\it Local projective model structures 
on simplicial presheaves},
K-Theory 24 (2001), no. 3, 283--301. 

\bibitem{Bo}
F. Borceux, {\it Handbook of categorical algebra. 2. Categories and
structures}, Encyclopedia of Mathematics and its Applications, 51.
Cambridge University Press, Cambridge, 1994. xviii+443 pp.

\bibitem{Bou}
A. K. Bousfield, {\it On the telescopic homotopy 
theory of spaces}, Trans. Amer. Math. Soc. 
353 (2001), no. 6, 2391--2426 (electronic).

\bibitem{Br}
K. S. Brown, {\it Abstract homotopy theory and generalized 
sheaf cohomology}, Trans. Amer. Math. Soc. 186 (1973), 419--458.

\bibitem{Co}
M. Cole, {\it Mixing model structures},
Topology Appl. 153 (2006), no. 7, 1016--1032. 

\bibitem{DI}
D. Dugger, D. C. Isaksen, {\it Weak equivalences of 
simplicial presheaves}, in Homotopy theory: 
relations with algebraic geometry, group 
cohomology, and algebraic K-theory, 97--113, 
Contemp. Math., 346, Amer. Math. Soc., 
Providence, RI, 2004. 

\bibitem{DHI}
D.  Dugger, S. Hollander, D. C. Isaksen, 
{\it Hypercovers and simplicial presheaves},
Math. Proc. Cambridge Philos. Soc. 136 (2004), no. 1, 
9--51.

\bibitem{Gi}
J. Giraud, {\it Cohomologie non ab\'{e}lienne},
Die Grundlehren der mathematischen Wissenschaften, 
Band 179. Springer-Verlag, Berlin-New York, 1971. ix+467 pp. 

\bibitem{GJ}
P.G. Goerss, J.F. Jardine, {\it Simplicial homotopy theory},
Progress in Mathematics, 174. Birkhauser Verlag, Basel, 1999.
xvi+510 pp.

\bibitem{Gr}
A. Grothendieck, {\it Rev\^etements \'etales et groupe fondamental},
S\'eminaire de G\'eom\'etrie Alg\'ebrique du Bois Marie 1960--1961 (SGA 1),
Lecture Notes in Mathematics, Vol. 224. Springer-Verlag, Berlin-New York, 
1971. xxii+447 pp. 

\bibitem{Hi}
P. S. Hirschhorn, {\it Model categories and their 
localizations}, Mathematical Surveys and 
Monographs, 99. American Mathematical
Society, Providence, RI, 2003. xvi+457 pp.

\bibitem{Ho1}
S. Hollander, {\it A homotopy theory for stacks}, Israel J. 
Math. 163 (2008), 93--124.

\bibitem{Ho2}
S. Hollander, {\it Characterizing algebraic stacks},
Proc. Amer. Math. Soc. 136 (2008), no. 4, 1465--1476.

\bibitem{HSS}
M. Hovey, B. Shipley, J. Smith, {\it Symmetric spectra}, 
J. Amer. Math. Soc. 13 (2000), no. 1, 149--208.

\bibitem{Is}
D. C. Isaksen, {\it Flasque model structures 
for simplicial presheaves},
K-Theory 36 (2005), no. 3--4, 371--395.

\bibitem{JT}
A. Joyal, M. Tierney, {\it Strong stacks and classifying spaces},
Category theory (Como, 1990), 213--236, Lecture Notes in Math.,
1488, Springer, Berlin, 1991.

\bibitem{JT2}
A. Joyal, M. Tierney, {\it Quasi-categories vs Segal spaces},
Categories in algebra, geometry and mathematical physics, 277--326, 
Contemp. Math., 431, Amer. Math. Soc., Providence, RI, 2007.

\bibitem{Vi}
A. Vistoli, {\it Grothendieck topologies, fibered categories 
and descent theory}, Fundamental algebraic geometry, 1--104, 
Math. Surveys Monogr., 123, Amer. Math. Soc., Providence, RI, 2005. 
\end{thebibliography}
\end{document}